\documentclass[12pt,a4paper]{article}

\usepackage[margin=2.5cm]{geometry}

\usepackage[T1]{fontenc}
\usepackage[utf8]{inputenc}
\usepackage{lmodern}
\usepackage{microtype}

\usepackage{amsmath,amssymb,amsthm}

\allowdisplaybreaks
\numberwithin{equation}{section}

\usepackage{xcolor}
\usepackage{cite}
\usepackage[bottom]{footmisc}

\usepackage[
colorlinks=true,
linkcolor=black,
citecolor=red,
urlcolor=blue
]{hyperref}

\theoremstyle{plain}

\newtheorem{theorem}{Theorem}[section]
\newtheorem{pro}[theorem]{Proposition}
\newtheorem{lem}[theorem]{Lemma}
\newtheorem{cor}[theorem]{Corollary}

\theoremstyle{definition}
\newtheorem{rem}[theorem]{Remark}

\date{}
\begin{document}
\title
{Stability of 2D MHD with Only Horizontal Dissipation near a Background Magnetic Field on $\mathbb{R}^2$}

\author{
	Jiahong Wu\footnote{Department of Mathematics, University of Notre Dame, Notre Dame, IN 46556, USA. Email: jwu29@nd.edu}
	\and
	Mengxin Yan\footnote{School of Mathematics, Shandong University, Jinan, Shandong 250000, P. R. China. Email: 202411867@mail.sdu.edu.cn}
	\and
	Ning Zhu\footnote{School of Mathematics, Shandong University, Jinan, Shandong 250000, P. R. China. Email: ning.zhu@sdu.edu.cn}
	\and
    Yi Zhu\footnote{Department of Mathematics,
	East China University of Science and Technology,
	Shanghai
	200237, P. R. China. Email: zhuyim@ecust.edu.cn}	
}

		\maketitle
	\begin{abstract}
	Whether the two-dimensional Navier--Stokes equations with dissipation
	in only one direction are stable in $H^2$ on the whole plane is a
	well-known open problem: without transverse dissipation, a cubic
	interaction fails to be time-integrable. We show that a magnetic field
	removes this obstruction. For the 2D incompressible MHD equations on
	$\mathbb R^2$ in which both the viscosity $\nu$ and the magnetic
	diffusivity $\mu$ act only in the horizontal direction, we prove the
	global nonlinear $H^2$ stability of every background field $\gamma e_2$
	aligned with the non-dissipative direction, with a stability threshold
	that is explicit in $\nu$, $\mu$ and $\gamma$. The stabilizing mechanism
	is genuinely nonlinear: at the linear level the background field
	produces no damping at all in the vertical direction. Instead, in
	Els\"asser variables the obstruction becomes an interaction between
	Alfv\'en waves travelling in opposite directions, and we identify
	nonlinear dissipative quantities that measure the overlap of these
	waves and are time-integrable although neither wave decays. We then
	sharpen the threshold by a space--time div--curl estimate for balance
	laws in the $(t,x_2)$-plane. For $\nu=\mu\leq1\leq\gamma$, the admissible
	perturbations are of size $\nu^{2/3}\gamma^{1/3}$, improving the direct
	scale $\nu$ and growing without bound as the background field
	strengthens.
\end{abstract}

\medskip
\noindent\textbf{2020 Mathematics Subject Classification.}
 35Q35, 35B35, 76E25.

\medskip
\noindent\textbf{Keywords.}
Incompressible MHD equations, horizontal dissipation, background magnetic
field, nonlinear stability, stability threshold, Els\"asser variables,
Alfv\'en waves, div--curl estimate.

\section{Introduction}

Consider the two-dimensional Navier--Stokes equations with dissipation in
the horizontal direction only,
\begin{align}\label{ns}
	\left\{
	\begin{aligned}
		&\partial_t u-\partial_1^2u+u\cdot\nabla u+\nabla p=0,
		\qquad x\in\mathbb R^2,\ t>0,\\
		&\nabla\cdot u=0,\\
		&u(x,0)=u_0(x).
	\end{aligned}
	\right.
\end{align}
Since the vorticity $\omega=\partial_1u_2-\partial_2u_1$ obeys a transport
equation with partial diffusion, its $L^p$ norms are non-increasing, and
smooth solutions of \eqref{ns} exist for all time. The difficulty lies
elsewhere: no uniform-in-time $H^2$ bound is known on $\mathbb R^2$, even
for small data, and the $H^2$ stability of the zero solution remains open
\cite{DongWuXuZhu2021}.

The obstruction is a single cubic interaction. The $H^1$ estimate for
$\omega$ reads
\begin{equation*}
	\frac{d}{dt}\|\nabla\omega\|_{L^2}^2
	+2\|\partial_1\nabla\omega\|_{L^2}^2
	=
	-2\sum_{j,k=1}^2\int_{\mathbb R^2}
	\partial_ku_j\,\partial_j\omega\,\partial_k\omega\,dx ,
\end{equation*}
and the term with $j=k=2$ is, by incompressibility,
\begin{equation}\label{bad-term-NS}
-2\int_{\mathbb R^2}\partial_2u_2\,(\partial_2\omega)^2\,dx
=
2\int_{\mathbb R^2}\partial_1u_1\,(\partial_2\omega)^2\,dx
\leq
2\|\partial_1u_1\|_{L^\infty}\|\partial_2\omega\|_{L^2}^2 .
\end{equation}
The factor $\partial_2\omega$ carries no horizontal derivative, so it is
merely bounded in time, while the horizontal dissipation controls
$\|\partial_1u\|_{L^\infty}$ only in $L^2_t$, not in $L^1_t$. On the strip
$\mathbb T\times\mathbb R$ a Poincar\'e inequality in the periodic
variable removes this obstruction; on $\mathbb R^2$ no such mechanism is
available.

The purpose of this paper is to show that coupling with a magnetic field
supplies exactly the transverse control that dissipation cannot. We
consider the two-dimensional incompressible magnetohydrodynamic (MHD)
equations with horizontal viscosity and horizontal magnetic diffusion,
\begin{align}\label{MHD-total-intro}
	\left\{
	\begin{aligned}
		&\partial_t U-\nu\partial_1^2U+U\cdot\nabla U-B\cdot\nabla B+\nabla P=0,
		\qquad x\in\mathbb R^2,~t>0,\\
		&\partial_t B-\mu\partial_1^2B+U\cdot\nabla B-B\cdot\nabla U=0,\\
		&\nabla\cdot U=\nabla\cdot B=0,\\
		&U(x,0)=U_0(x),~B(x,0)=B_0(x),
	\end{aligned}
	\right.
\end{align}
where $U=(U_1,U_2)$ is the velocity, $B=(B_1,B_2)$ the magnetic field, $P$
the total pressure, and $\nu>0$, $\mu>0$ are the viscosity and the
magnetic diffusivity. The MHD equations govern electrically conducting
fluids such as plasmas and liquid metals, in which dissipative processes are
frequently strongly anisotropic; \eqref{MHD-total-intro} is the extreme
configuration in which both dissipations act in the same single direction.

For fully dissipative 2D MHD, global regularity is classical
\cite{DuvautLions1972,SermangeTemam1983}, and many partially dissipative
variants have been studied; see, for instance,
\cite{CaoWu2011,RenWuXiangZhang2014,CaoWuYuan2014,LinXuZhang2015}. For
\eqref{MHD-total-intro} on $\mathbb R^2$, Cao, Regmi and Wu
\cite{CaoRegmiWu2013} obtained global a priori bounds and regularity
criteria, while global regularity for general data remains open. On
$\mathbb T\times\mathbb R$, Paicu and Zhu \cite{PaicuZhu2021} proved global
regularity for small data, relying on the Poincar\'e inequality in the
periodic direction.

We study the stability of the steady state
$U^{(0)}=0$, $B^{(0)}=\gamma e_2=(0,\gamma)$. Since
$(U,B)\mapsto(U,-B)$ leaves the system invariant, we may assume
$\gamma>0$. Writing $U=u$, $B=\gamma e_2+b$, the perturbation satisfies
\begin{align}\label{MHD-horizontal dissipation}
	\left\{
	\begin{aligned}
		&\partial_t u-\nu\partial_1^2u+u\cdot\nabla u-b\cdot\nabla b+\nabla p
		=\gamma\partial_2b,\\
		&\partial_t b-\mu\partial_1^2b+u\cdot\nabla b-b\cdot\nabla u
		=\gamma\partial_2u,\\
		&\nabla\cdot u=\nabla\cdot b=0,\\
		&u(x,0)=u_0(x),\qquad b(x,0)=b_0(x).
	\end{aligned}
	\right.
\end{align}
We set
\[
\nu_m=\min\{\nu,\mu\},
\qquad
\nu_M=\max\{\nu,\mu\}.
\]
The background field is aligned with the non-dissipative direction
$x_2$. The linear terms $\gamma\partial_2b$ and $\gamma\partial_2u$ provide
no dissipation; instead, they generate wave propagation along $x_2$, which
is the direction left uncontrolled by the dissipation.

This wave structure can be made explicit. Let $\mathbb P$ be the Leray
projection onto divergence-free vector fields and set
\[
N_u\triangleq\mathbb P\big(b\cdot\nabla b-u\cdot\nabla u\big),
\qquad
N_b\triangleq b\cdot\nabla u-u\cdot\nabla b,
\]
so that \eqref{MHD-horizontal dissipation} reads
$\partial_tu-\nu\partial_1^2u-\gamma\partial_2b=N_u$ and
$\partial_tb-\mu\partial_1^2b-\gamma\partial_2u=N_b$. Applying
$\partial_t-\mu\partial_1^2$ to the first equation and
$\partial_t-\nu\partial_1^2$ to the second, and using the other equation to
eliminate the coupling term, we find that $u$ and $b$ satisfy the
wave-type equations
\begin{align}\label{wave}
	\left\{
	\begin{aligned}
		&\partial_t^2u-(\nu+\mu)\partial_1^2\partial_tu+\nu\mu\partial_1^4u
		-\gamma^2\partial_2^2u
		=(\partial_t-\mu\partial_1^2)N_u+\gamma\partial_2N_b,\\
		&\partial_t^2b-(\nu+\mu)\partial_1^2\partial_tb+\nu\mu\partial_1^4b
		-\gamma^2\partial_2^2b
		=(\partial_t-\nu\partial_1^2)N_b+\gamma\partial_2N_u.
	\end{aligned}
	\right.
\end{align}
The background field thus produces the term $-\gamma^2\partial_2^2$, which
acts precisely in the direction that lacks dissipation. In earlier work on
the stabilizing effect of a background magnetic field, see for instance
\cite{JiWu2020,BoardmanLinWu2020,LinJiWuYan2020,LinWuZhu2025}, a wave
structure of this
kind was used to generate additional smoothing and decay.

When we try to use \eqref{wave} directly in the energy estimates, however,
we encounter a basic difficulty. The natural way to extract control of
$\partial_2u$ is to use the equation for $b$ to trade the coupling term
$\gamma\partial_2u$ for a time derivative, that is, to consider the cross
quantity $(b,\partial_2u)_{L^2}$. A direct computation gives
\begin{equation}\label{cross-term}
	\begin{split}
	\gamma\|\partial_2u\|_{L^2}^2-\gamma\|\partial_2b\|_{L^2}^2
	={}&\frac{d}{dt}(b,\partial_2u)_{L^2}
	-\mu(\partial_1^2b,\partial_2u)_{L^2}
	+\nu(\partial_2b,\partial_1^2u)_{L^2}\\
	&-(N_b,\partial_2u)_{L^2}
	+(\partial_2b,N_u)_{L^2}.
	\end{split}
\end{equation}
Thus the gain $\gamma\|\partial_2u\|_{L^2}^2$ comes with the loss
$-\gamma\|\partial_2b\|_{L^2}^2$, which could only be absorbed by
dissipation of $b$ in the $x_2$-direction; but the equation for $b$ has no
vertical diffusion either. The symmetric computation, based on
$(u,\partial_2b)_{L^2}$, produces $-\gamma\|\partial_2u\|_{L^2}^2$ in the
same way. The same fact is visible at the linear level: the eigenvalues
$\lambda$ of the linearization of \eqref{wave} at frequency
$\xi=(\xi_1,\xi_2)$ solve
$\lambda^2+(\nu+\mu)\xi_1^2\lambda+\nu\mu\xi_1^4+\gamma^2\xi_2^2=0$, and one
checks that
\[
-\nu_M\xi_1^2\leq\operatorname{Re}\lambda\leq-\nu_m\xi_1^2 .
\]
The damping rate of every Fourier mode is therefore governed by $\xi_1$
alone, exactly as in the absence of the magnetic field, and the modes with
$\xi_1=0$ merely oscillate. The background field produces no additional
linear dissipation, and its stabilizing effect must be sought at the
nonlinear level.

The left-hand side of \eqref{cross-term} indicates how to proceed. It
equals $\gamma(\partial_2(u+b),\partial_2(u-b))_{L^2}$, a correlation
between the two combinations $u\pm b$. This leads us to the Els\"asser
variables
\[
\Omega^\pm=u\pm b,
\]
which diagonalize the linear coupling. In terms of $\Omega^\pm$, system
\eqref{MHD-horizontal dissipation} becomes
\begin{align}\label{Elsasser-system}
	\left\{
	\begin{aligned}
		&\partial_t\Omega^+
		-\nu_+\partial_1^2\Omega^+
		-\nu_-\partial_1^2\Omega^-
		+\Omega^-\cdot\nabla\Omega^+
		+\nabla p
		=\gamma\partial_2\Omega^+,\\
		&\partial_t\Omega^-
		-\nu_+\partial_1^2\Omega^-
		-\nu_-\partial_1^2\Omega^+
		+\Omega^+\cdot\nabla\Omega^-
		+\nabla p
		=-\gamma\partial_2\Omega^-,\\
		&\nabla\cdot\Omega^+
		=\nabla\cdot\Omega^-=0,\\
		&\Omega^\pm(x,0)=\Omega^\pm_0(x),
	\end{aligned}
	\right.
\end{align}
where
\[
\nu_+=\frac{\nu+\mu}{2},
\qquad
\nu_-=\frac{\nu-\mu}{2},
\qquad
\Omega^\pm_0=u_0\pm b_0.
\]
If the dissipative and nonlinear terms are neglected, then
$(\partial_t-\gamma\partial_2)\Omega^+=0$ and
$(\partial_t+\gamma\partial_2)\Omega^-=0$: the two Els\"asser modes are
Alfv\'en waves travelling in opposite directions along the background
field. Moreover, the nonlinearity couples $\Omega^+$ only with $\Omega^-$
(this includes the pressure, since
$-\Delta p=\nabla\cdot\nabla\cdot(\Omega^-\otimes\Omega^+)$); there is no
self-interaction.

Even in these variables, the Navier--Stokes obstruction reappears. In the
$H^2$ estimate for \eqref{Elsasser-system}, most cubic terms contain at
least two horizontal derivatives, either directly or through the
divergence-free identities $\partial_2\Omega^\pm_2=-\partial_1\Omega^\pm_1$,
and can be bounded by the energy and the dissipation. The exception is
\[
N_{51}=\big(\partial_2^2\Omega^-_1\,\partial_1\Omega^+_1,\;
\partial_2^2\Omega^+_1\big)_{L^2},
\]
which, like \eqref{bad-term-NS}, contains only one horizontal derivative:
two of its three factors are merely bounded in time. What distinguishes
$N_{51}$ from \eqref{bad-term-NS} is that these two factors belong to
waves travelling in opposite directions. We exploit this in two steps.

First, we use the equation for $\Omega^-$ to replace the transverse
derivative $\gamma\partial_2\Omega^-$ by $-\partial_t\Omega^-$ plus
dissipative and nonlinear terms, integrate by parts in time, and then use
the equation for $\Omega^+$ to convert the resulting time derivative of
$\Omega^+$ back into a transverse derivative. Because $\Omega^+$ and
$\Omega^-$ propagate in opposite directions, the principal part of $N_{51}$
reappears with the sign that allows it to be moved to the left-hand side;
for co-propagating waves the two contributions would cancel and the
identity would carry no information (see the proof of
Proposition~\ref{prop:N51}). A related substitution, in which the
equation for $b_2$ is used to replace $\partial_2u_2$ in the most
dangerous cubic term, was used in \cite[Proposition~2.2]{DingPanZhu2025}.

Second, this procedure, as well as the treatment of the remaining parts of
the critical terms, leaves quartic space--time quantities that are
controlled neither by the energy nor by the viscous dissipation, namely
\begin{align}\label{def-E1}
	E_1^\pm(t)
	\triangleq
	\sum_{k=0}^1
	\int_0^t\int_{\mathbb R^2}
	|\partial_2^k\Omega_1^\pm|^2
	|\partial_2^2\Omega_1^\mp|^2
	\,dx\,d\tau .
\end{align}
These are the nonlinear dissipative quantities of the problem. They
measure the overlap of two counter-propagating waves, and since such waves
overlap only for a finite time, they remain bounded although neither
factor decays. To prove this, we integrate the density
$|\partial_2^k\Omega^+_1|^2$ in $x_2$ from $-\infty$ to obtain a weight
$W(x_2,t)$ and consider
$Q(t)=\int_{\mathbb R^2}W\,|\partial_2^2\Omega^-_1|^2\,dx$. Because the two
waves travel in opposite directions, the interaction density
$\int_{\mathbb R^2}|\partial_2^k\Omega^+_1|^2|\partial_2^2\Omega^-_1|^2\,dx$
appears in $Q'(t)$ with a definite sign, while $Q(t)$ itself is bounded by
the energy. Hence $E_1^\pm$ is controlled by the $H^2$ energy up to
higher-order terms. Interaction quantities of this type were introduced
by Ding, Pan and Zhu \cite{DingPanZhu2025} for the viscous non-resistive
MHD system on $\mathbb R^2$. There, the space--time integrals
$\int_0^t\!\int_{\mathbb R^2}|u_1\pm b_1|^2|\partial_1^2b_2|^2\,dx\,d\tau$,
built from the first components of the Els\"asser variables, are
controlled through the weight
$\int_{-\infty}^{x_2}|u_1\pm b_1|^2\,dy_2$, and this is what allows
global well-posedness for data that are small only in $H^2$. We extend
this idea to the full Els\"asser system \eqref{Elsasser-system}, in
which both the viscosity and the magnetic diffusion act only
horizontally. The construction is also in the spirit of Glimm's
interaction potential for hyperbolic systems \cite{Glimm1965} and is
related to the weighted energy estimates for Alfv\'en waves in
\cite{HeXuYu2018,CaiLei2018,WeiZhang2017}. Coupling $E_1^\pm$ to the $H^2$
energy, with the weight $\nu_m^{-1}$, closes the a priori estimates and
yields our first result.

\begin{theorem}\label{thm}
	Let $\nu,\mu,\gamma>0$, and let $(u_0,b_0)\in H^2(\mathbb R^2)$ be
	divergence free. There exist absolute constants $c_0, C_0>0$ such that,
	if
	\[
	\|u_0\|_{H^2}+\|b_0\|_{H^2}
	\leq
	c_0\min\left\{
	\nu_m,
	\frac{\nu_m^{\frac12}}{\nu_M^{\frac12}}\gamma
	\right\},
	\]
	then the MHD system \eqref{MHD-horizontal dissipation} admits a unique
	global solution $(u,b)\in C([0,\infty);H^2(\mathbb R^2))$ with
	$\partial_1u,\partial_1b\in L^2(0,\infty;H^2(\mathbb R^2))$, satisfying
	\[
	\begin{aligned}
	&\sup_{t\geq0}
	\left(
	\|u(t)\|_{H^2}+\|b(t)\|_{H^2}
	\right)+
	\left(	\int_0^\infty
	\big(\nu \|\partial_1 u (t )\|_{H^2}^2
	+ \mu \|\partial_1b (t)\|_{H^2}^2\big)
	\,dt\right)^{\frac12}\\
	&\qquad\leq
	C_0
	\left(
	\|u_0\|_{H^2}+\|b_0\|_{H^2}
	\right).
	\end{aligned}
	\]
\end{theorem}

Theorem \ref{thm} shows that a background magnetic field aligned with the
non-dissipative direction stabilizes the flow on the whole plane, in sharp
contrast with the Navier--Stokes problem \eqref{ns}. The mechanism is the
one described above: the Els\"asser formulation turns the bad interaction
into an interaction between counter-propagating waves, and the nonlinear
dissipative quantities $E_1^\pm$ quantify their separation. Local
well-posedness and uniqueness are discussed in Remark~\ref{rem:local}.

Our second aim is to improve the stability threshold of
Theorem~\ref{thm}. For $\gamma=1$ and $\nu=\mu\ll1$, Theorem~\ref{thm}
allows perturbations of size $O(\nu)$. This scale is dictated by the cubic
terms with two horizontal derivatives, such as
\[
N_{11}=\big(\partial_1^2\Omega^-_2\,\partial_2\Omega^+,\;
\partial_1^2\Omega^+\big)_{L^2}.
\]
Since the factor $\partial_2\Omega^+$ carries no horizontal derivative,
estimating $N_{11}$ factor by factor forces both remaining factors to be
paid for by the dissipation, and one obtains, schematically,
\[
\int_0^t|N_{11}(\tau)|\,d\tau
\lesssim
\int_0^t
\|(\Omega^+|\Omega^-)\|_{H^2}
\|(\partial_1\Omega^+|\partial_1\Omega^-)\|_{H^2}^2\,d\tau
\lesssim
\nu_m^{-1}E^{\frac32}(t),
\]
where $E(t)$ is the total energy introduced in Section~\ref{THM}. Since
$E^{\frac12}$ measures the size of the perturbation, this leads to the
scale $E^{\frac12}\lesssim\nu_m$. However, the two factors
$\partial_1^2\Omega^-$ and $\partial_2\Omega^+$ once more belong to
counter-propagating waves, and products of this type can be controlled in
space--time by a bilinear estimate introduced recently by Jiang, Ren and
Zhou \cite{JiangRenZhou2026} for the fully dissipative MHD equations in a
strong background field. After integration in the variable transverse to
the field, here $x_1$, suitable quadratic quantities built from $(u,b)$
satisfy a pair of balance laws in the $(t,x_2)$-plane,
\[
\partial_tf_{11}+\partial_2f_{12}=G_1,\qquad
\partial_tf_{21}-\partial_2f_{22}=G_2,
\]
with the algebraic identity
\[
f_{11}f_{22}+f_{12}f_{21}
=\frac18\Big(\|K\Omega^+\|_{L^2_{x_1}}^2\|L\Omega^-\|_{L^2_{x_1}}^2
+\|K\Omega^-\|_{L^2_{x_1}}^2\|L\Omega^+\|_{L^2_{x_1}}^2\Big)
\]
for suitable differential operators $L$ and $K$, for instance
$(L,K)=(\partial_1^2,\partial_2)$ for $N_{11}$. The vector field
$(f_{11},f_{12})$ is divergence free in the $(t,x_2)$-plane up to $G_1$,
the vector field $(f_{22},f_{21})$ is curl free up to $G_2$, and the
left-hand side is their Euclidean inner product; this is the div--curl
structure familiar from the compensated compactness theory of Murat and
Tartar \cite{Murat1978,Tartar1979}. A quantitative one-dimensional
div--curl lemma, due to Lai, Shao and Zhou \cite{LaiShaoZhou2025} (see
also \cite{WangZhou2023}) and used in \cite{JiangRenZhou2026}, bounds the
space--time integral of $f_{11}f_{22}+f_{12}f_{21}$ by
$\sup_t\|f_{11}\|_{L^1}\|f_{21}\|_{L^1}$ and the source terms; see
Lemma~\ref{div-curl lemma}. Its proof again uses an interaction potential,
so that both of our results rest on the same principle, applied to
pointwise densities for Theorem~\ref{thm} and to quadratic quantities
averaged in $x_1$ for the result below. The new point in the present
setting is that only horizontal dissipation is available: the source
terms $G_1$ and $G_2$ have to be controlled by
$\|\partial_1(u,b)\|_{H^2}$ alone, which is possible because every
operator $L$ we use contains a horizontal derivative. As a consequence,
only one factor of $N_{11}$ has to be paid for by the dissipation, and the
singular factor $\nu_m^{-1}$ in front of the leading cubic term is reduced
to $\nu_M^{\frac14}\nu_m^{-\frac34}$. This yields the following improved
stability threshold when the viscosity and the magnetic diffusivity are
bounded by one and the background magnetic field is of order one or
larger.

\begin{theorem}\label{cor}
	Let $(u_0,b_0)\in H^2(\mathbb R^2)$ be divergence free. Assume that
	$0<\nu,\mu\leq1$ and $\gamma\geq1$. There exist absolute constants
	$c_0, C_0>0$ such that, if
	\begin{align}\label{initial data}
		\|u_0\|_{H^2}+\|b_0\|_{H^2}
		\leq
		c_0\min\left\{
		\nu_m^{\frac23}\gamma^{\frac13},
		\nu_m^{\frac34}\gamma^{\frac12}
		\right\},
	\end{align}
	then the MHD system \eqref{MHD-horizontal dissipation} admits a unique
	global solution $(u,b)\in C([0,\infty);H^2(\mathbb R^2))$ with
	$\partial_1u,\partial_1b\in L^2(0,\infty;H^2(\mathbb R^2))$, satisfying
	\[
	\begin{aligned}
	&\sup_{t\geq0}
	\left(
	\|u(t)\|_{H^2}+\|b(t)\|_{H^2}
	\right)+
	\left(	\int_0^\infty
	\big(\nu \|\partial_1 u (t )\|_{H^2}^2
	+ \mu \|\partial_1b (t)\|_{H^2}^2\big)
	\,dt\right)^{\frac12}\\
	&\qquad\leq
	C_0
	\left(
	\|u_0\|_{H^2}+\|b_0\|_{H^2}
	\right).
	\end{aligned}
	\]
\end{theorem}

The sharper parameter-dependent criterion established in the proof of
Theorem \ref{cor} yields the following consequence: under a suitable
comparability condition on the viscosity and magnetic diffusivity, the
stability threshold takes a particularly simple form.

\begin{cor}\label{cor1}
	Let $(u_0,b_0)\in H^2(\mathbb R^2)$ be divergence free. Assume that
	$0<\nu,\mu\leq1$, $\gamma\geq1$, and
	\[
	\nu_M\leq\beta\nu_m^{\frac13}
	\]
	for some $\beta>0$. There exist a constant $c_\beta>0$, depending only
	on $\beta$, and an absolute constant $C_0>0$ such that, if
	\[
	\|u_0\|_{H^2}+\|b_0\|_{H^2}
	\leq
	c_\beta\,\nu_m^{\frac23}\gamma^{\frac13},
	\]
	then the conclusions of Theorem \ref{cor} hold. In particular, this is
	the case with $\beta=1$ when $\nu=\mu$.
\end{cor}

\begin{rem}\label{rem:exponent}
	Theorem \ref{cor} and Corollary \ref{cor1} show that a sharper
	exploitation of the propagation structure induced by the background
	magnetic field lowers the viscosity exponent in the stability threshold.
	For $\gamma=1$, the basic estimate in Theorem \ref{thm} gives an
	admissible perturbation size of order $O(\nu_m)$, whereas
	Theorem \ref{cor} improves this to $O(\nu_m^{\frac34})$.
	Under the additional comparability condition
	$\nu_M\lesssim\nu_m^{\frac13}$, Corollary \ref{cor1} further improves
	the scale to $O(\nu_m^{\frac23})$. Since smaller viscosity exponents
	correspond to larger admissible perturbations when $\nu_m\ll1$, these
	results demonstrate a substantial improvement of the stability threshold.
\end{rem}

\begin{rem}\label{rem:gamma}
	Theorem \ref{cor} and Corollary \ref{cor1} also quantify the stabilizing
	effect of the strength of the background magnetic field. For fixed
	viscosity and magnetic diffusivity, the thresholds in
	Theorem \ref{cor} and Corollary \ref{cor1} contain positive powers of
	$\gamma$ and behave like $\nu_m^{\frac23}\gamma^{\frac13}$ as
	$\gamma\to\infty$. Thus a sufficiently strong background magnetic field
	allows the system to remain globally stable for perturbations that are
	large in $H^2$. These perturbations nevertheless remain small relative
	to the background field, since
	$\nu_m^{\frac23}\gamma^{\frac13}/\gamma\to0$ as $\gamma\to\infty$. By
	contrast, the threshold of Theorem \ref{thm} saturates at $\nu_m$ as
	$\gamma\to\infty$. This is in line with Kraichnan's prediction
	\cite{Kraichnan1965} that a strong background field weakens the
	nonlinear interactions. For the fully dissipative MHD system, Jiang,
	Ren and Zhou \cite{JiangRenZhou2026} obtained global solutions for
	perturbations of size up to order $\gamma^{\frac13}$ for fixed
	dissipation coefficients (see \cite[Remark~2.1]{JiangRenZhou2026});
	Theorem \ref{cor} exhibits the same power of $\gamma$ with dissipation
	in only one direction, and with an explicit dependence on $\nu$ and
	$\mu$.
\end{rem}

\begin{rem}\label{rem:general}
	The restrictions $0<\nu,\mu\leq1$ and $\gamma\geq1$ in
	Theorem \ref{cor} are used only to simplify the form of the threshold.
	The proof in Section~\ref{COR} shows that, for all $\nu,\mu,\gamma>0$,
	the conclusions of Theorem \ref{cor} hold under the condition
	\[
	\|u_0\|_{H^2}+\|b_0\|_{H^2}
	\leq
	c_0\min\left\{
	\nu_m^{\frac23}\gamma^{\frac13},\;
	\frac{\nu_m^{\frac34}}{\nu_M^{\frac14}}\gamma^{\frac12},\;
	\frac{\nu_m^{\frac12}}{\nu_M^{\frac12}}\gamma
	\right\};
	\]
	see \eqref{general-gamma-sharp}. This threshold is never smaller than
	the one in Theorem \ref{thm}: writing
	$a=\nu_m$ and $d=\gamma\nu_m^{\frac12}\nu_M^{-\frac12}$, we have
	\[
	\min\{a,d\}\leq a^{\frac23}d^{\frac13}\leq\nu_m^{\frac23}\gamma^{\frac13},
	\qquad
	\min\{a,d\}\leq a^{\frac12}d^{\frac12}
	=\gamma^{\frac12}\nu_m^{\frac34}\nu_M^{-\frac14},
	\qquad
	\min\{a,d\}\leq d .
	\]
	In particular, if $\nu=\mu\leq\gamma$, the threshold reduces to
	$\nu^{\frac23}\gamma^{\frac13}$.
\end{rem}

The stabilizing effect of a background magnetic field on electrically
conducting fluids has been studied extensively. In two dimensions, global
small solutions near a background field for the system without magnetic
diffusion were constructed by Lin, Xu and Zhang \cite{LinXuZhang2015} and
by Ren, Wu, Xiang and Zhang \cite{RenWuXiangZhang2014} (see
\cite{RenXiangZhang2016} for a strip); in \cite{RenWuXiangZhang2014} the
data are also assumed to be small in anisotropic Sobolev spaces with
negative indices, and Ding, Pan and Zhu \cite{DingPanZhu2025} recently
showed that smallness in $H^2$ alone suffices. Zhang \cite{Zhang2016} treated
large background fields, and Ji and Wu \cite{JiWu2020} the resistive
system without viscosity. Further anisotropic or partially damped
configurations were considered in
\cite{BoardmanLinWu2020,LinJiWuYan2020,LaiWuZhang2021}, and on
$\mathbb T\times\mathbb R$ Suo and Jiu \cite{SuoJiu2022} proved global
well-posedness near a background field for the non-resistive system with
horizontal viscosity, using symmetry and Poincar\'e-type inequalities. In
three dimensions, we mention the works of Xu and Zhang
\cite{XuZhang2015}, Abidi and Zhang \cite{AbidiZhang2017}, Deng and Zhang
\cite{DengZhang2018}, and Pan, Zhou and Zhu \cite{PanZhouZhu2018} on the
system without magnetic diffusion, and those of Wu and Zhu
\cite{WuZhu2021}, Lin, Wu and Zhu \cite{LinWuZhu2023,LinWuZhu2025}, and
Lai, Wu, Zhang and Zhao \cite{LaiWuZhangZhao2026} on systems with
anisotropic partial dissipation. For ideal or weakly dissipative MHD, the
global stability of Alfv\'en waves in a strong magnetic field goes back to
Bardos, Sulem and Sulem \cite{BardosSulemSulem1988} and has been developed
in \cite{HeXuYu2018,CaiLei2018,WeiZhang2017}. Global solutions with large
perturbations in the strong-field regime, as predicted by Kraichnan
\cite{Kraichnan1965}, were obtained in \cite{BardosSulemSulem1988} for
ideal MHD, in \cite{Zhang2016,JiangJiang2023} for the viscous
non-resistive system, and in \cite{CaiCuiJiangLiu2026,JiangRenZhou2026}
for the fully dissipative system. The quantitative dependence of the
admissible perturbation size on the dissipation coefficients, the
\emph{stability threshold}, is a central theme in hydrodynamic stability
\cite{BedrossianGermainMasmoudi2019}, including for MHD shear flows
\cite{Liss2020,Dolce2024}.

The present setting differs from the works above in three respects. In
contrast with the results on $\mathbb T\times\mathbb R$, the domain is the
whole plane, so no Poincar\'e-type mechanism is available. In contrast
with \cite{CaiCuiJiangLiu2026,JiangRenZhou2026} and the other stability
results above, both dissipations act in the same single direction, so the
system has no direct dissipation in $x_2$, and, as shown above, the
background field does not create any at the linear level. Finally, the
thresholds are explicit in all three parameters $\nu$, $\mu$ and
$\gamma$.

\medskip
\noindent\textbf{Notations.}
Throughout the paper, all spatial functions are defined on $\mathbb R^2$
unless otherwise specified. For $1\leq p\leq\infty$, we write
$L^p=L^p(\mathbb R^2)$, and, for any nonnegative integer $k$, we write
$H^k=H^k(\mathbb R^2)$. The notation $(f,g)_{L^2}$ denotes the
$L^2(\mathbb R^2)$ inner product.

For $1\leq p,q<\infty$, the mixed norm is defined by
\[
\|f\|_{L^p_{x_1}L^q_{x_2}}
\triangleq
\left(
\int_{\mathbb R}
\left(
\int_{\mathbb R}|f(x_1,x_2)|^q\,dx_2
\right)^{\frac{p}{q}}
dx_1
\right)^{\frac1p},
\]
with the usual modifications when $p=\infty$ or $q=\infty$. Thus, the
norm in the rightmost variable is taken first. For a function of
$(x_1,x_2)$, $\|f\|_{L^p_{x_1}}$ denotes the $L^p$ norm in $x_1$ at fixed
$x_2$, which is a function of $x_2$; similarly for $\|f\|_{L^p_{x_2}}$. We
use $L^p_x$ and $H^k_x$ when the spatial variables need to be distinguished
from time. For a Banach space $X$ and $1\leq r\leq\infty$, we set
\[
\|f\|_{L^r_tX}
\triangleq
\|f\|_{L^r(0,t;X)}.
\]

For two functions $f$ and $g$, we use the shorthand
\[
\|(f|g)\|_X
\triangleq
\|f\|_X+\|g\|_X.
\]
Finally, $a\lesssim b$ means that $a\leq Cb$ for an absolute constant
$C>0$, independent of $\nu$, $\mu$, $\gamma$ and $t$. The value of $C$ may
change from line to line.

\vskip .1in
The rest of the paper is organized as follows. Section~\ref{THM}
performs the scaling reduction, introduces the energy functionals, and
establishes the basic energy and interaction estimates, leading to the
proof of Theorem~\ref{thm}. Section~\ref{COR} develops a refined
space--time estimate based on the div--curl structure and proves
Theorem~\ref{cor}, Remark~\ref{rem:general} and Corollary~\ref{cor1}. The
div--curl lemma is proved in Appendix~\ref{app:divcurl}.

\section{Proof of Theorem \ref{thm}}\label{THM}

We begin with a scaling reduction to the normalized background field. For
$\gamma>0$, set
\[
v(x,t)=\frac1\gamma u\left(x,\frac{t}{\gamma}\right),\qquad
h(x,t)=\frac1\gamma b\left(x,\frac{t}{\gamma}\right),\qquad
q(x,t)=\frac1{\gamma^2}p\left(x,\frac{t}{\gamma}\right).
\]
Then $(v,h)$ satisfies
\begin{align}\label{scaling system}
\left\{
\begin{aligned}
&\partial_t v-\frac{\nu}{\gamma}\partial_1^2v
+v\cdot\nabla v-h\cdot\nabla h+\nabla q=\partial_2h,\\
&\partial_t h-\frac{\mu}{\gamma}\partial_1^2h
+v\cdot\nabla h-h\cdot\nabla v=\partial_2v,\\
&\nabla\cdot v=\nabla\cdot h=0,\\
&v(x,0)=\frac1\gamma u_0(x),\qquad
h(x,0)=\frac1\gamma b_0(x).
\end{aligned}
\right.
\end{align}
Thus the background field is normalized to one, while the dissipative
coefficients become
\[
\widetilde\nu=\frac{\nu}{\gamma},\qquad
\widetilde\mu=\frac{\mu}{\gamma},\qquad
\widetilde\nu_m=\frac{\nu_m}{\gamma},\qquad
\widetilde\nu_M=\frac{\nu_M}{\gamma}.
\]
It is therefore enough to establish the a priori estimates for the
normalized system. In the energy estimates below, we set $\gamma=1$ and,
for notational simplicity, suppress the tildes on the dissipative
parameters. In particular, the Els\"asser variables $\Omega^\pm=u\pm b$
satisfy \eqref{Elsasser-system} with $\gamma=1$.

\begin{rem}[Local well-posedness and uniqueness]\label{rem:local}
	Let $\nu,\mu,\gamma>0$ and let $(u_0,b_0)\in H^2$ be divergence free.
	Then \eqref{MHD-horizontal dissipation} has a unique solution
	$(u,b)\in C([0,T];H^2)$ with
	$\partial_1u,\partial_1b\in L^2(0,T;H^2)$ on an interval $[0,T]$ whose
	length depends only on $\nu_m$ and $\|(u_0|b_0)\|_{H^2}$, and the
	solution can be continued as long as $\|(u,b)(t)\|_{H^2}$ remains
	bounded. Local existence follows from standard procedures, see, e.g.,
	\cite{MajdaBertozzi2002}: one adds $\varepsilon\Delta u$ and
	$\varepsilon\Delta b$ to the equations, derives $H^2$ bounds that are
	uniform in $\varepsilon$ on a short time interval, and passes to the
	limit. Since $H^2$ is a borderline space for the two-dimensional Euler
	equations, the horizontal dissipation is used here: the trilinear
	terms are bounded by anisotropic interpolation as in
	Section~\ref{Estimate of energy E_0(t)}, and the resulting powers of
	$\|(\partial_1\Omega^+|\partial_1\Omega^-)\|_{H^2}$ are absorbed by the
	dissipation through Young's inequality. The linear terms
	$\gamma\partial_2b$ and $\gamma\partial_2u$ cancel in these estimates,
	so $\gamma$ does not enter. Uniqueness at the $H^2$ level is
	elementary: if $\delta$ denotes the difference of two such solutions,
	the worst terms in the $L^2$ estimate for $\delta$ are bounded by
	\[
	\int_{\mathbb R^2}|\nabla(u,b)|\,|\delta|^2\,dx
	\leq\|\nabla(u,b)\|_{L^2_{x_1}L^\infty_{x_2}}
	\|\delta\|_{L^\infty_{x_1}L^2_{x_2}}\|\delta\|_{L^2}
	\lesssim\|(u|b)\|_{H^2}\|\delta\|_{L^2}^{\frac32}
	\|\partial_1\delta\|_{L^2}^{\frac12},
	\]
	and Young's and Gronwall's inequalities give $\delta\equiv0$. For the
	two-dimensional Navier--Stokes equations with dissipation in only one
	direction, uniqueness even holds in the class of $H^1$ solutions, a
	level that is heuristically optimal \cite{ZhouWu2026}. In view of these
	facts, it suffices to prove the global a priori estimates below for
	smooth solutions that decay at infinity (for instance with data in
	$H^4$); the constants do not depend on the additional regularity, and
	the general case follows by density. This also justifies the time
	differentiations and the integrations over half-lines in $x_2$
	performed below.
\end{rem}

We introduce the energy
\begin{align*}
	E_0(t)
	\triangleq{}&
	\sup_{0\leq\tau\leq t}
	\left(
	\|\Omega^+(\tau)\|_{H^2}^2
	+\|\Omega^-(\tau)\|_{H^2}^2
	\right)
	+
	\int_0^t
	\left(
	\nu\|\partial_1u(\tau)\|_{H^2}^2
	+\mu\|\partial_1b(\tau)\|_{H^2}^2
	\right)\,d\tau,
\end{align*}
where the dissipative part is written in terms of $(u,b)$ to retain the
exact weights $\nu$ and $\mu$, and, with the interaction functionals
$E_1^\pm$ defined in \eqref{def-E1}, the total energy
\[
E(t)
\triangleq
E_0(t)
+\nu_m^{-1}E_1^+(t)
+\nu_m^{-1}E_1^-(t).
\]
Note that $E(0)=E_0(0)=2\big(\|u_0\|_{H^2}^2+\|b_0\|_{H^2}^2\big)$. The
proof of Theorem \ref{thm} is reduced to the following two estimates:
\begin{align}
	E_0(t)
	\leq{}&
	C E(0)
	+C\nu_m^{-1}\big(E_1^+(t)+E_1^-(t)\big)
	+C\frac{1+\nu_M^{\frac12}\nu_m^{\frac12}}{\nu_m}
	E^{\frac32}(t)
	+C\nu_m^{-1}E^2(t),
	\label{E_0(t)}
	\\
	E_1^\pm(t)
	\leq{}&
	C\frac{\nu_M^{\frac12}}{\nu_m^{\frac12}}E^2(t)
	+C\nu_m^{-1}E^{\frac52}(t)
	+C\nu_m E^{\frac32}(t).
	\label{E1^{pm}(t)}
\end{align}
Here and below, $C>0$ denotes a generic absolute constant.

The rest of this section is organized as follows. In
Section~\ref{Estimate of energy E_0(t)} we derive the $H^2$ energy
identity for $\Omega^\pm$ and reduce \eqref{E_0(t)} to six cubic terms
$N_1,\dots,N_6$. The terms $N_1$--$N_4$ contain enough horizontal
derivatives and are estimated directly. In $N_5$ and $N_6$, all
contributions but one are bounded by the energy, the dissipation and the
interaction functionals $E_1^\pm$; the remaining critical term $N_{51}$,
the analogue of \eqref{bad-term-NS}, is isolated in
Proposition~\ref{prop:N51}. Section~\ref{subsec:N51} proves this
proposition by exploiting the counter-propagation of $\Omega^\pm$: we
replace $\partial_2\Omega^-$ by means of the equation for $\Omega^-$,
integrate by parts in time, and use the equation for $\Omega^+$, after
which the principal term reappears with the coefficient $-2$ and can be
solved for. Section~\ref{Estimate of energy E_1(t)} proves
\eqref{E1^{pm}(t)}: the interaction densities are generated by the time
derivative of a weighted functional, whose weight is obtained by
integrating $|\partial_2^k\Omega^+_1|^2$ in $x_2$, following the idea of
\cite{DingPanZhu2025}. Finally,
Section~\ref{subsec:close2} combines \eqref{E_0(t)} and
\eqref{E1^{pm}(t)}, closes the estimate by a continuity argument, and
returns to general $\gamma$ by scaling.

\subsection{Estimate for $E_0(t)$}\label{Estimate of energy E_0(t)}

By the norm equivalence
\[
\Vert(\Omega^+|\Omega^-)\Vert_{H^2}^2
\sim \Vert(\Omega^+|\Omega^-)\Vert_{L^2}^2
+\sum_{m=1}^{2}\Vert(\partial_m^2\Omega^+|\partial_m^2\Omega^-)\Vert_{L^2}^2,
\]
it suffices to estimate the $L^2$ norms and the pure second derivatives.

Taking the $L^2$ inner products of the equations in \eqref{Elsasser-system} with $\Omega^+$ and $\Omega^-$, respectively, and using $\nabla\cdot\Omega^\pm=0$, we obtain
\[
\begin{aligned}
	&\Vert\Omega^+(t)\Vert_{L^2}^2+\Vert\Omega^-(t)\Vert_{L^2}^2
	+2\int_0^t\Big[\nu_+\big(\Vert\partial_1\Omega^+\Vert_{L^2}^2+\Vert\partial_1\Omega^-\Vert_{L^2}^2\big)
	+2\nu_-(\partial_1\Omega^-,\partial_1\Omega^+)_{L^2}\Big]\,d\tau\\
	&=\Vert\Omega^+_0\Vert_{L^2}^2+\Vert\Omega^-_0\Vert_{L^2}^2.
\end{aligned}
\]
By the definitions of $\nu_+$ and $\nu_-$,
\[\begin{aligned}
\nu_+\big(\|\partial_1\Omega^+\|_{L^2}^2+\|\partial_1\Omega^-\|_{L^2}^2\big)
+2\nu_-(\partial_1\Omega^-,\partial_1\Omega^+)_{L^2}
&=2\nu\|\partial_1u\|_{L^2}^2+2\mu\|\partial_1b\|_{L^2}^2
\\
&\geq\nu_m\big(\|\partial_1\Omega^+\|_{L^2}^2+\|\partial_1\Omega^-\|_{L^2}^2\big).
\end{aligned}\]
Therefore,
\begin{align}\label{L^2 energy estimate}
	\Vert\Omega^+(t)\Vert_{L^2}^2+\Vert\Omega^-(t)\Vert_{L^2}^2
	+4\int_0^t\big(\nu\Vert\partial_1u(\tau)\Vert_{L^2}^2+\mu\Vert\partial_1b(\tau)\Vert_{L^2}^2\big)\,d\tau
	= \Vert\Omega^+_0\Vert_{L^2}^2+\Vert\Omega^-_0\Vert_{L^2}^2 ,
\end{align}
and, in particular,
\[
2\nu_m\int_0^t\big(\Vert\partial_1\Omega^+(\tau)\Vert_{L^2}^2+\Vert\partial_1\Omega^-(\tau)\Vert_{L^2}^2\big)\,d\tau
\leq \Vert\Omega^+_0\Vert_{L^2}^2+\Vert\Omega^-_0\Vert_{L^2}^2.
\]

We next estimate the pure second derivatives. We apply $\partial_i^2$, $i=1,2$, to the equations for $\Omega^\pm$ and take the $L^2$ inner products with $\partial_i^2\Omega^\pm$. The transport terms $\pm\partial_2\Omega^\pm$ and the pressure do not contribute, the terms $\Omega^\mp\cdot\nabla\partial_i^2\Omega^\pm$ vanish by incompressibility, and for each $i=1,2$ the diffusion form satisfies
\[
\begin{aligned}
	&\quad\nu_+\bigl(\|\partial_1\partial_i^2\Omega^+\|_{L^2}^2
	+\|\partial_1\partial_i^2\Omega^-\|_{L^2}^2\bigr)
	+2\nu_-(\partial_1\partial_i^2\Omega^-,
	\partial_1\partial_i^2\Omega^+)_{L^2}\\
	&=2\nu\|\partial_1\partial_i^2u\|_{L^2}^2
	+2\mu\|\partial_1\partial_i^2b\|_{L^2}^2.
\end{aligned}
\]
Combining with \eqref{L^2 energy estimate}, we obtain
\begin{align}\label{E0-preliminary}
	E_0(t)\lesssim E(0)+\sum_{j=1}^{6} \sup_{0\leq s\leq t}\left|\int_0^s N_j(\tau)\,d\tau\right|,
\end{align}
where
\[
\begin{aligned}
	N_1&\triangleq(\partial_1^2(\Omega^-\cdot\nabla\Omega^+),\partial_1^2\Omega^+)_{L^2},
	&\quad N_2&\triangleq(\partial_2^2(\Omega^-\cdot\nabla\Omega^+_2),\partial_2^2\Omega^+_2)_{L^2},\\
	N_3&\triangleq(\partial_1^2(\Omega^+\cdot\nabla\Omega^-),\partial_1^2\Omega^-)_{L^2},
	&\quad N_4&\triangleq(\partial_2^2(\Omega^+\cdot\nabla\Omega^-_2),\partial_2^2\Omega^-_2)_{L^2},\\
	N_5&\triangleq(\partial_2^2(\Omega^-\cdot\nabla\Omega^+_1),\partial_2^2\Omega^+_1)_{L^2},
	&\quad N_6&\triangleq(\partial_2^2(\Omega^+\cdot\nabla\Omega^-_1),\partial_2^2\Omega^-_1)_{L^2}.
\end{aligned}
\]

The terms $N_1-N_4$ can be treated directly. Indeed, $N_1$ and $N_3$ contain sufficient derivatives in the dissipative $x_1$-direction, while for $N_2$ and $N_4$ the divergence-free identities $\partial_2\Omega^\pm_2=-\partial_1\Omega^\pm_1$ provide the required $x_1$-derivatives. Hence, H\"older's inequality and anisotropic interpolation yield
\begin{align}\label{N_1234}
	\sum_{j=1}^{4}\int_0^t|N_j(\tau)|\,d\tau
	\lesssim \int_0^t\Vert(\Omega^+|\Omega^-)\Vert_{H^2}\Vert(\partial_1\Omega^+|\partial_1\Omega^-)\Vert_{H^2}^2\,d\tau
	\lesssim\nu_m^{-1}E_0^{\frac32}(t).
\end{align}

It remains to estimate $N_5$ and $N_6$. For $N_5$, expanding the derivatives gives
\[
\begin{aligned}
	N_5={}&(\partial_2^2\Omega^-_1\partial_1\Omega^+_1,\partial_2^2\Omega^+_1)_{L^2}
	+(\partial_2^2\Omega^-_2\partial_2\Omega^+_1,\partial_2^2\Omega^+_1)_{L^2}\\
	&+2(\partial_2\Omega^-_1\partial_1\partial_2\Omega^+_1,\partial_2^2\Omega^+_1)_{L^2}
	+2(\partial_2\Omega^-_2\partial_2^2\Omega^+_1,\partial_2^2\Omega^+_1)_{L^2}.
\end{aligned}
\]
Using $\partial_2\Omega^-_2=-\partial_1\Omega^-_1$ and integration by parts, we further write
\[
\begin{aligned}
	N_5={}&(\partial_2^2\Omega^-_1\partial_1\Omega^+_1,\partial_2^2\Omega^+_1)_{L^2}
	+(\partial_2\Omega^-_1\partial_1\partial_2\Omega^+_1,\partial_2^2\Omega^+_1)_{L^2}\\
	&-(\partial_2(\partial_2\Omega^-_1\partial_2\Omega^+_1),\partial_1\partial_2\Omega^+_1)_{L^2}
	+2(\partial_2\Omega^-_1\partial_1\partial_2\Omega^+_1,\partial_2^2\Omega^+_1)_{L^2}\\
	&+2(\Omega^-_1\partial_1\partial_2^2\Omega^+_1,\partial_2^2\Omega^+_1)_{L^2}
	+2(\Omega^-_1\partial_2^2\Omega^+_1,\partial_1\partial_2^2\Omega^+_1)_{L^2}\\
	\triangleq{}&N_{51}+N_{52}+N_{53}+N_{54}+N_{55}+N_{56}.
\end{aligned}
\]

For $j=2,\ldots,6$, H\"older's inequality and the definition of $E_1^\pm$ give
\begin{align}\label{N_{52}}
	\begin{split}
		&\quad\int_0^t|N_{5j}(\tau)|\,d\tau\\
		&\lesssim
		\Bigg(\int_0^t\sum_{k=0}^1\int_{\mathbb R^2}\big(|\partial_2^k\Omega_1^-|^2|\partial_2^2\Omega_1^+|^2
		+|\partial_2^k\Omega_1^+|^2|\partial_2^2\Omega_1^-|^2\big)\,dx\,d\tau\Bigg)^{\frac12}
		\left(\int_0^t\|\partial_1\Omega_1^+\|_{H^2}^2\,d\tau\right)^{\frac12}\\
		&\lesssim\nu_m^{-\frac12}E_0^{\frac12}(t)\big(E_1^+(t)+E_1^-(t)\big)^{\frac12}.
	\end{split}
\end{align}

The only remaining difficulty is $N_{51}$. A direct estimate does not provide sufficient time integrability, and its treatment requires a further use of the propagation structure of the Els\"asser system. We isolate the required estimate in the following proposition, whose proof is postponed to the next subsection.

\begin{pro}[Estimate for the critical term $N_{51}$]
	\label{prop:N51}
	Let $(\Omega^+,\Omega^-)$ be a smooth solution of the Els\"asser
	system \eqref{Elsasser-system} with $\gamma=1$ on $[0,T]$. Then, for every
	$t\in[0,T]$,
	\[
	\left|\int_0^t N_{51}(\tau)\,d\tau\right|
	\lesssim
	\nu_m^{-1}E_0^2(t)
	+\nu_M^{\frac12}\nu_m^{-\frac12}E_0^{\frac32}(t)
	+\nu_m^{-\frac12}E_0^{\frac12}(t)
	\bigl(E_1^+(t)+E_1^-(t)\bigr)^{\frac12}.
	\]
\end{pro}

Combining Proposition \ref{prop:N51} with \eqref{N_{52}}, we obtain
\begin{align}\label{N_5}
	\left|\int_0^t N_5(\tau)\,d\tau\right|
	\lesssim\nu_m^{-1}E_0^2(t)+\nu_M^{\frac12}\nu_m^{-\frac12}E_0^{\frac32}(t)
	+\nu_m^{-\frac12}E_0^{\frac12}(t)\big(E_1^+(t)+E_1^-(t)\big)^{\frac12}.
\end{align}
By symmetry,
\begin{align}\label{N_6}
	\left|\int_0^t N_6(\tau)\,d\tau\right|
	\lesssim\nu_m^{-1}E_0^2(t)+\nu_M^{\frac12}\nu_m^{-\frac12}E_0^{\frac32}(t)
	+\nu_m^{-\frac12}E_0^{\frac12}(t)\big(E_1^+(t)+E_1^-(t)\big)^{\frac12}.
\end{align}
Collecting \eqref{E0-preliminary}, \eqref{N_1234}, \eqref{N_5}, and \eqref{N_6}, we obtain
\begin{align}\label{E_0-qian}
	E_0(t)\lesssim E(0)+\frac{1+\nu_M^{\frac12}\nu_m^{\frac12}}{\nu_m}E^{\frac32}(t)
	+\nu_m^{-1}E^2(t)
	+\nu_m^{-\frac12}E_0^{\frac12}(t)\big(E_1^+(t)+E_1^-(t)\big)^{\frac12}.
\end{align}
By Young's inequality,
\[
C\nu_m^{-\frac12}E_0^{\frac12}(t)\big(E_1^+(t)+E_1^-(t)\big)^{\frac12}
\leq\frac12E_0(t)+C\nu_m^{-1}\big(E_1^+(t)+E_1^-(t)\big).
\]
Absorbing the first term into the left-hand side of \eqref{E_0-qian} yields \eqref{E_0(t)}.

\subsection{Proof of Proposition \ref{prop:N51}}\label{subsec:N51}

\begin{proof}
	Integrating by parts in $x_1$, we obtain
	\begin{equation}\label{first-second term}
		N_{51}=-(\partial_1\partial_2^2\Omega^-_1\Omega^+_1,\partial_2^2\Omega^+_1)_{L^2}
		-(\partial_2^2\Omega^-_1\Omega^+_1,\partial_1\partial_2^2\Omega^+_1)_{L^2}.
	\end{equation}
	The second term on the right-hand side can be estimated directly:
	\begin{align}\label{N51firstterm}
		\begin{split}
			\int_0^t\big|(\partial_2^2\Omega^-_1\Omega^+_1,\partial_1\partial_2^2\Omega^+_1)_{L^2}\big|\,d\tau
			&\lesssim
			\left(\int_0^t\|\partial_1\partial_2^2\Omega^+_1\|_{L^2}^2\,d\tau\right)^{\frac12}
			\left(\int_0^t\int_{\mathbb R^2}|\Omega^+_1|^2|\partial_2^2\Omega^-_1|^2\,dx\,d\tau\right)^{\frac12}\\
			&\lesssim\nu_m^{-\frac12}E_0^{\frac12}(t)\big(E_1^+(t)+E_1^-(t)\big)^{\frac12}.
		\end{split}
	\end{align}
	
	We now focus on the first term in \eqref{first-second term}. From the equation for $\Omega^-$,
	\[
	-\partial_2\Omega^-=\partial_t\Omega^-+\Omega^+\cdot\nabla\Omega^--\nu_+\partial_1^2\Omega^--\nu_-\partial_1^2\Omega^++\nabla p.
	\]
	It follows that
	\begin{align}\label{prop1-eq1-part1}
		\begin{split}
			-(\partial_1\partial_2^2\Omega^-_1\Omega^+_1,\partial_2^2\Omega^+_1)_{L^2}
			={}&-(\partial_1\partial_2(-\partial_t\Omega^-_1)\Omega^+_1,\partial_2^2\Omega^+_1)_{L^2}\\
			&-(\partial_1\partial_2(-\Omega^+\cdot\nabla\Omega^-_1)\Omega^+_1,\partial_2^2\Omega^+_1)_{L^2}\\
			&-(\partial_1\partial_2(\nu_+\partial_1^2\Omega^-_1+\nu_-\partial_1^2\Omega^+_1)\Omega^+_1,\partial_2^2\Omega^+_1)_{L^2}\\
			&-(\partial_1\partial_2(-\partial_1p)\Omega^+_1,\partial_2^2\Omega^+_1)_{L^2}\\
			\triangleq{}&N_{511}+N_{512}+N_{513}+N_{514}.
		\end{split}
	\end{align}
	
	We shall repeatedly use
	\begin{align}\label{nu and mu}
		\nu_+\partial_1^2\Omega^\pm+\nu_-\partial_1^2\Omega^\mp=\nu\partial_1^2u\pm\mu\partial_1^2b,
	\end{align}
	and hence
	\[
	\|\nu_+\partial_1^2\Omega^\pm+\nu_-\partial_1^2\Omega^\mp\|_{L^2}
	\lesssim\nu\|\partial_1^2u\|_{L^2}+\mu\|\partial_1^2b\|_{L^2}.
	\]
	
	For $N_{511}$, integration by parts in time gives
	\[
	N_{511}=\frac{d}{dt}(\partial_1\partial_2\Omega^-_1\Omega^+_1,\partial_2^2\Omega^+_1)_{L^2}
	-(\partial_1\partial_2\Omega^-_1,\partial_t(\Omega^+_1\partial_2^2\Omega^+_1))_{L^2}
	\triangleq N_{511}^1+N_{511}^2.
	\]
	Using
	\[
	-\partial_t\Omega^+=-\partial_2\Omega^++\Omega^-\cdot\nabla\Omega^+-\nu_+\partial_1^2\Omega^+-\nu_-\partial_1^2\Omega^-+\nabla p,
	\]
	we obtain
	\[
	\begin{aligned}
		-\partial_t(\Omega^+_1\partial_2^2\Omega^+_1)
		={}&-\partial_2(\Omega^+_1\partial_2^2\Omega^+_1)
		+\big(\Omega^-\cdot\nabla\Omega^+_1-\nu_+\partial_1^2\Omega^+_1-\nu_-\partial_1^2\Omega^-_1+\partial_1p\big)\partial_2^2\Omega^+_1\\
		&+\partial_2^2\big(\Omega^-\cdot\nabla\Omega^+_1-\nu_+\partial_1^2\Omega^+_1-\nu_-\partial_1^2\Omega^-_1+\partial_1p\big)\Omega^+_1.
	\end{aligned}
	\]
	Therefore,
	\[
	\begin{aligned}
		N_{511}^2
		={}&(\partial_1\partial_2^2\Omega^-_1\Omega^+_1,\partial_2^2\Omega^+_1)_{L^2}\\
		&+(\partial_1\partial_2\Omega^-_1,(\Omega^-\cdot\nabla\Omega^+_1-\nu_+\partial_1^2\Omega^+_1-\nu_-\partial_1^2\Omega^-_1+\partial_1p)\partial_2^2\Omega^+_1)_{L^2}\\
		&+(\partial_1\partial_2\Omega^-_1,\partial_2^2(\Omega^-\cdot\nabla\Omega^+_1-\nu_+\partial_1^2\Omega^+_1-\nu_-\partial_1^2\Omega^-_1+\partial_1p)\Omega^+_1)_{L^2}\\
		\triangleq{}&N_{511}^{21}+N_{511}^{22}+N_{511}^{23}.
	\end{aligned}
	\]
	Integrating by parts in $x_2$ shows that
	\[
	N_{511}^{21}=(\partial_1\partial_2^2\Omega^-_1\,\Omega^+_1,\partial_2^2\Omega^+_1)_{L^2}\triangleq\mathcal T ,
	\]
	that is, $N_{511}^{21}$ coincides with \emph{minus} the left-hand side of \eqref{prop1-eq1-part1}. Writing $-\mathcal T=N_{511}+\dots+N_{514}$ and $N_{511}=N_{511}^1+\mathcal T+N_{511}^{22}+N_{511}^{23}$, we obtain
	\[
	-2\mathcal T=N_{511}^1+N_{511}^{22}+N_{511}^{23}+N_{512}+N_{513}+N_{514},
	\]
	so it suffices to bound the time integrals of the terms on the right. The coefficient of $\mathcal T$ is $-2\neq0$; this is where the counter-propagation of $\Omega^\pm$ enters (for co-propagating waves the two occurrences of $\mathcal T$ would cancel).
	
	For $N_{511}^{22}$, by H\"older's inequality, anisotropic interpolation, \eqref{nu and mu}, and
	\[
	-\Delta p=\nabla\cdot\nabla\cdot(\Omega^-\otimes\Omega^+),
	\]
	we have
	\begin{align}\label{N_{511}^{22}}
		\begin{split}
			&\quad|N_{511}^{22}|\\
			&\lesssim
			\|\partial_1\partial_2\Omega^-_1\|_{L^2_{x_1}L^\infty_{x_2}}
			\|\partial_2^2\Omega^+_1\|_{L^\infty_{x_1}L^2_{x_2}}
			\|\Omega^-\cdot\nabla\Omega^+_1-\nu_+\partial_1^2\Omega^+_1-\nu_-\partial_1^2\Omega^-_1+\partial_1p\|_{L^2}\\
			&\lesssim
			\|\partial_1\partial_2\Omega^-_1\|_{L^2}^{\frac12}
			\|\partial_1\partial_2^2\Omega^-_1\|_{L^2}^{\frac12}
			\|\partial_2^2\Omega^+_1\|_{L^2}^{\frac12}
			\|\partial_1\partial_2^2\Omega^+_1\|_{L^2}^{\frac12}\\
			&\qquad\times\big(\nu\|\partial_1^2u\|_{L^2}+\mu\|\partial_1^2b\|_{L^2}
			+\|\partial_1(\Omega^-\otimes\Omega^+)\|_{L^2}  + \Vert \Omega^-_2 \Vert_{L^2_{x_1}L^\infty_{x_2}} \Vert \partial_2 \Omega^{+} _1 \Vert_{L^\infty_{x_1}L^2_{x_2}} \big)\\
			&\lesssim
			\|(\Omega^-|\Omega^+)\|_{H^2}^2\|(\partial_1\Omega^-|\partial_1\Omega^+)\|_{H^2}^2+\nu\|(\Omega^-|\Omega^+)\|_{H^2}\|\partial_1u\|_{H^2}\|(\partial_1\Omega^-|\partial_1\Omega^+)\|_{H^2}\\
			&\qquad+\mu\|(\Omega^-|\Omega^+)\|_{H^2}\|\partial_1b\|_{H^2}\|(\partial_1\Omega^-|\partial_1\Omega^+)\|_{H^2}.
		\end{split}
	\end{align}
	
	For $N_{511}^{23}$, integration by parts yields
	\[
	\begin{aligned}
		N_{511}^{23}
		={}&-(\partial_1\partial_2^2\Omega^-_1,\partial_2(\Omega^-\cdot\nabla\Omega^+_1-\nu_+\partial_1^2\Omega^+_1-\nu_-\partial_1^2\Omega^-_1+\partial_1p)\Omega^+_1)_{L^2}\\
		&-(\partial_1\partial_2\Omega^-_1,\partial_2(\Omega^-\cdot\nabla\Omega^+_1-\nu_+\partial_1^2\Omega^+_1-\nu_-\partial_1^2\Omega^-_1+\partial_1p)\partial_2\Omega^+_1)_{L^2}.
	\end{aligned}
	\]
	Thus, by H\"older's inequality, Sobolev embedding, interpolation, and \eqref{nu and mu},
	\begin{align}\label{N_{511}^{23}}
		\begin{split}
			&\quad|N_{511}^{23}|\\
			&\lesssim
			\|\partial_1\partial_2^2\Omega^-_1\|_{L^2}\|\Omega^+_1\|_{L^\infty}
			\|\partial_2(\Omega^-\cdot\nabla\Omega^+_1-\nu_+\partial_1^2\Omega^+_1-\nu_-\partial_1^2\Omega^-_1+\partial_1p)\|_{L^2}\\
			&\quad+
			\|\partial_1\partial_2\Omega^-_1\|_{L^2_{x_1}L^\infty_{x_2}}
			\|\partial_2\Omega^+_1\|_{L^\infty_{x_1}L^2_{x_2}}
			\|\partial_2(\Omega^-\cdot\nabla\Omega^+_1-\nu_+\partial_1^2\Omega^+_1-\nu_-\partial_1^2\Omega^-_1+\partial_1p)\|_{L^2}\\
			&\lesssim
			\|\partial_1\partial_2^2\Omega^-_1\|_{L^2}
			\|\Omega^+_1\|_{L^2}^{\frac14}\|\partial_1\Omega^+_1\|_{L^2}^{\frac14}
			\|\partial_2\Omega^+_1\|_{L^2}^{\frac14}\|\partial_1\partial_2\Omega^+_1\|_{L^2}^{\frac14}\\
			&\qquad\times\big(\nu\|\partial_2\partial_1^2u\|_{L^2}+\mu\|\partial_2\partial_1^2b\|_{L^2}
			+\|\partial_2\nabla\cdot(\Omega^-\otimes\Omega^+)\|_{L^2}\big)\\
			&\quad+
			\|\partial_1\partial_2\Omega^-_1\|_{L^2}^{\frac12}
			\|\partial_1\partial_2^2\Omega^-_1\|_{L^2}^{\frac12}
			\|\partial_2\Omega^+_1\|_{L^2}^{\frac12}
			\|\partial_1\partial_2\Omega^+_1\|_{L^2}^{\frac12}\\
			&\qquad\times\big(\nu\|\partial_2\partial_1^2u\|_{L^2}+\mu\|\partial_2\partial_1^2b\|_{L^2}
			+\|\partial_2\nabla\cdot(\Omega^-\otimes\Omega^+)\|_{L^2}\big)\\
			&\lesssim
			\|(\Omega^-|\Omega^+)\|_{H^2}^2\|(\partial_1\Omega^-|\partial_1\Omega^+)\|_{H^2}^2\\
			&\qquad+\nu\|(\Omega^-|\Omega^+)\|_{H^2}\|\partial_1u\|_{H^2}\|(\partial_1\Omega^-|\partial_1\Omega^+)\|_{H^2}+\mu\|(\Omega^-|\Omega^+)\|_{H^2}\|\partial_1b\|_{H^2}\|(\partial_1\Omega^-|\partial_1\Omega^+)\|_{H^2}.
		\end{split}
	\end{align}
	
	The time boundary term satisfies
	\[
	\left|(\partial_1\partial_2\Omega^-_1\Omega^+_1,\partial_2^2\Omega^+_1)_{L^2}\right|
	\lesssim E_0^{\frac32}(t).
	\]
	Consequently, \eqref{N_{511}^{22}} and \eqref{N_{511}^{23}} yield
	\begin{align}\label{511}
		\left|\int_0^tN_{511}\,d\tau-\int_0^t(\partial_1\partial_2^2\Omega^-_1\Omega^+_1,\partial_2^2\Omega^+_1)_{L^2}\,d\tau\right|
		\lesssim E_0^{\frac32}(t)+\nu_m^{-1}E_0^2(t)+\nu_M^{\frac12}\nu_m^{-\frac12}E_0^{\frac32}(t).
	\end{align}
	
	For $N_{512}$, integration by parts gives
	\[
	N_{512}=(\partial_2(-\Omega^+\cdot\nabla\Omega^-_1)\Omega^+_1,\partial_1\partial_2^2\Omega^+_1)_{L^2}
	+(\partial_2(-\Omega^+\cdot\nabla\Omega^-_1)\partial_1\Omega^+_1,\partial_2^2\Omega^+_1)_{L^2},
	\]
	and therefore
	\begin{align}\label{512}
		\begin{split}
			|N_{512}|
			&\lesssim\|\partial_2(\Omega^+\cdot\nabla\Omega^-_1)\|_{L^2}
			\Big(\|\Omega^+_1\|_{L^\infty}\|\partial_1\partial_2^2\Omega^+_1\|_{L^2}
			+\|\partial_1\Omega^+_1\|_{L^\infty}\|\partial_2^2\Omega^+_1\|_{L^2}\Big)\\
			&\lesssim\|(\Omega^-|\Omega^+)\|_{H^2}^2\|(\partial_1\Omega^-|\partial_1\Omega^+)\|_{H^2}^2.
		\end{split}
	\end{align}
	
	For $N_{513}$, by \eqref{nu and mu} and interpolation,
	\begin{align}\label{513}
		\begin{split}
			|N_{513}|
			&\lesssim
			\big(\nu\|\partial_2\partial_1^2u\|_{L^2}+\mu\|\partial_2\partial_1^2b\|_{L^2}\big)
			\Big(\|\partial_1\Omega^+_1\|_{L^\infty}\|\partial_2^2\Omega^+_1\|_{L^2}
			+\|\Omega^+_1\|_{L^\infty}\|\partial_1\partial_2^2\Omega^+_1\|_{L^2}\Big)\\
			&\lesssim
			\nu\|\partial_2\partial_1^2u\|_{L^2}\|(\Omega^-|\Omega^+)\|_{H^2}\|(\partial_1\Omega^-|\partial_1\Omega^+)\|_{H^2}\\
			&\qquad+\mu\|\partial_2\partial_1^2b\|_{L^2}\|(\Omega^-|\Omega^+)\|_{H^2}\|(\partial_1\Omega^-|\partial_1\Omega^+)\|_{H^2}.
		\end{split}
	\end{align}
	
	Finally, using the pressure equation,
	\begin{align}\label{514}
		\begin{split}
			|N_{514}|
			&\lesssim\|\partial_2\nabla\cdot(\Omega^-\otimes\Omega^+)\|_{L^2}
			\Big(\|\Omega^+_1\|_{L^\infty}\|\partial_1\partial_2^2\Omega^+_1\|_{L^2}
			+\|\partial_1\Omega^+_1\|_{L^\infty}\|\partial_2^2\Omega^+_1\|_{L^2}\Big)\\
			&\lesssim\|(\Omega^-|\Omega^+)\|_{H^2}^2\|(\partial_1\Omega^-|\partial_1\Omega^+)\|_{H^2}^2.
		\end{split}
	\end{align}
	
	Combining \eqref{511}--\eqref{514} with the identity for $-2\mathcal T$, we obtain
	\begin{align}\label{N51secondterm}
		\left|\int_0^t(\partial_1\partial_2^2\Omega^-_1\Omega^+_1,\partial_2^2\Omega^+_1)_{L^2}\,d\tau\right|
		\lesssim E_0^{\frac32}(t)+\nu_m^{-1}E_0^2(t)+\nu_M^{\frac12}\nu_m^{-\frac12}E_0^{\frac32}(t).
	\end{align}
	Since $\nu_M\geq\nu_m$, the first term on the right-hand side is controlled by the last one. Combining \eqref{N51secondterm} with \eqref{N51firstterm} completes the proof of Proposition \ref{prop:N51}.
\end{proof}

\subsection{Estimate for $E_1^\pm(t)$}\label{Estimate of energy E_1(t)}

Having established \eqref{E_0(t)}, we now turn to the interaction energies $E_1^\pm(t)$, which capture the nonlinear interactions involving the transverse derivatives. By symmetry, it suffices to estimate $E_1^+(t)$; the estimate for $E_1^-(t)$ follows in the same way.

For $k\in\{0,1\}$, the first component of the equation for $\Omega^+$ gives
\[
\frac12\partial_2|\partial_2^k\Omega_1^+|^2
=\frac12\partial_t|\partial_2^k\Omega_1^+|^2
+\partial_2^k(\Omega^-\cdot\nabla\Omega_1^+)\partial_2^k\Omega_1^+
+\partial_1\partial_2^kp\,\partial_2^k\Omega_1^+
-\big(\nu_+\partial_1^2\partial_2^k\Omega^+_1+\nu_-\partial_1^2\partial_2^k\Omega^-_1\big)\partial_2^k\Omega_1^+.
\]
Integrating this identity with respect to $y_2$ from $-\infty$ to $x_2$, multiplying by $|\partial_2^2\Omega_1^-|^2$, and then integrating over $\mathbb R^2$, we obtain
\begin{align}\label{key-e1}
	\begin{split}
		\frac12\int_{\mathbb R^2}|\partial_2^k\Omega_1^+|^2|\partial_2^2\Omega_1^-|^2\,dx
		={}&\frac12\int_{\mathbb R^2}\int_{-\infty}^{x_2}\partial_t|\partial_2^k\Omega_1^+|^2\,dy_2\,|\partial_2^2\Omega_1^-|^2\,dx\\
		&+\int_{\mathbb R^2}\int_{-\infty}^{x_2}\partial_2^k(\Omega^-\cdot\nabla\Omega_1^+)\partial_2^k\Omega_1^+\,dy_2\,|\partial_2^2\Omega_1^-|^2\,dx\\
		&+\int_{\mathbb R^2}\int_{-\infty}^{x_2}\partial_1\partial_2^kp\,\partial_2^k\Omega_1^+\,dy_2\,|\partial_2^2\Omega_1^-|^2\,dx\\
		&-\int_{\mathbb R^2}\int_{-\infty}^{x_2}\big(\nu_+\partial_1^2\partial_2^k\Omega^+_1+\nu_-\partial_1^2\partial_2^k\Omega^-_1\big)
		\partial_2^k\Omega_1^+\,dy_2\,|\partial_2^2\Omega_1^-|^2\,dx\\
		\triangleq{}&{J_1+J_2+J_3+J_4.}
	\end{split}
\end{align}

We first estimate $J_1$. Integration by parts in time gives
\[
J_1=\frac12\frac{d}{dt}\int_{\mathbb R^2}\int_{-\infty}^{x_2}|\partial_2^k\Omega_1^+|^2\,dy_2\,|\partial_2^2\Omega_1^-|^2\,dx
-\frac12\int_{\mathbb R^2}\int_{-\infty}^{x_2}|\partial_2^k\Omega_1^+|^2\,dy_2\,\partial_t|\partial_2^2\Omega_1^-|^2\,dx
\triangleq J_{11}+J_{12}.
\]
For $J_{11}$, H\"older's inequality and interpolation yield
\begin{align}\label{J_{11}}
	\begin{split}
		\left|\int_0^tJ_{11}(\tau)\,d\tau\right|
		&\lesssim\sup_{0\leq\tau\leq t}\|\partial_2^k\Omega_1^+\|_{L^\infty_{x_1}L^2_{x_2}}^2\|\partial_2^2\Omega_1^-\|_{L^2}^2\\
		&\lesssim\sup_{0\leq\tau\leq t}\|\partial_2^k\Omega_1^+\|_{L^2}\|\partial_1\partial_2^k\Omega_1^+\|_{L^2}\|\partial_2^2\Omega_1^-\|_{L^2}^2
		\lesssim E_0^2(t).
	\end{split}
\end{align}

For $J_{12}$, the equation for $\Omega^-$ gives
\[
\frac12\partial_t|\partial_2^2\Omega_1^-|^2
=-\frac12\partial_2|\partial_2^2\Omega_1^-|^2
-\partial_2^2(\Omega^+\cdot\nabla\Omega_1^-)\partial_2^2\Omega_1^-
-\partial_1\partial_2^2p\,\partial_2^2\Omega_1^-
+\big(\nu_+\partial_1^2\partial_2^2\Omega^-_1+\nu_-\partial_1^2\partial_2^2\Omega^+_1\big)\partial_2^2\Omega_1^-.
\]
Therefore,
\[
\begin{aligned}
	J_{12}={}&\frac12\int_{\mathbb R^2}\int_{-\infty}^{x_2}|\partial_2^k\Omega_1^+|^2\,dy_2\,\partial_2|\partial_2^2\Omega_1^-|^2\,dx\\
	&+\int_{\mathbb R^2}\int_{-\infty}^{x_2}|\partial_2^k\Omega_1^+|^2\,dy_2\,\partial_2^2(\Omega^+\cdot\nabla\Omega_1^-)\partial_2^2\Omega_1^-\,dx\\
	&+\int_{\mathbb R^2}\int_{-\infty}^{x_2}|\partial_2^k\Omega_1^+|^2\,dy_2\,\partial_1\partial_2^2p\,\partial_2^2\Omega_1^-\,dx\\
	&-\int_{\mathbb R^2}\int_{-\infty}^{x_2}|\partial_2^k\Omega_1^+|^2\,dy_2\,
	\big(\nu_+\partial_1^2\partial_2^2\Omega^-_1+\nu_-\partial_1^2\partial_2^2\Omega^+_1\big)\partial_2^2\Omega_1^-\,dx\\
	\triangleq{}&J_{121}+J_{122}+J_{123}+J_{124}.
\end{aligned}
\]

For $J_{121}$, integration by parts in $x_2$ gives
\[
J_{121}=-\frac12\int_{\mathbb R^2}|\partial_2^k\Omega_1^+|^2|\partial_2^2\Omega_1^-|^2\,dx,
\]
which can be absorbed into the left-hand side of \eqref{key-e1}.

For $J_{122}$, expanding the derivatives and using $\nabla\cdot\Omega^+=0$, we obtain
\[
\begin{aligned}
	J_{122}
	={}&\int_{\mathbb R^2}\int_{-\infty}^{x_2}|\partial_2^k\Omega_1^+|^2\,dy_2\,
	\partial_2^2\Omega^+\cdot\nabla\Omega_1^-\,\partial_2^2\Omega_1^-\,dx\\
	&+2\int_{\mathbb R^2}\int_{-\infty}^{x_2}|\partial_2^k\Omega_1^+|^2\,dy_2\,
	\partial_2\Omega^+\cdot\nabla\partial_2\Omega_1^-\,\partial_2^2\Omega_1^-\,dx\\
	&-\frac12\int_{\mathbb R^2}\int_{-\infty}^{x_2}\partial_1|\partial_2^k\Omega_1^+|^2\,dy_2\,
	\Omega_1^+|\partial_2^2\Omega_1^-|^2\,dx
	-\frac12\int_{\mathbb R^2}|\partial_2^k\Omega_1^+|^2\Omega_2^+|\partial_2^2\Omega_1^-|^2\,dx\\
	\triangleq{}&J_{1221}+J_{1222}+J_{1223}+J_{1224}.
\end{aligned}
\]

For $J_{1221}$, H\"older's inequality and anisotropic interpolation imply
\begin{align}\label{J1221}
	\begin{split}
		|J_{1221}|
		&\lesssim\|\partial_2^k\Omega_1^+\|_{L^\infty_{x_1}L^2_{x_2}}^2
		\Big(\|\partial_1\Omega_1^-\|_{L^\infty}\|\Omega_1^+\|_{H^2}\|\Omega_1^-\|_{H^2}\\
		&\qquad\qquad+\|\partial_1\partial_2\Omega_1^+\|_{L^\infty_{x_1}L^2_{x_2}}
		\|\partial_2\Omega_1^-\|_{L^2_{x_1}L^\infty_{x_2}}\|\Omega_1^-\|_{H^2}\Big)\\
		&\lesssim\big(\|\partial_1\Omega^+\|_{H^2}^2+\|\partial_1\Omega^-\|_{H^2}^2\big)
		\big(\|\Omega^+\|_{H^2}^3+\|\Omega^-\|_{H^2}^3\big).
	\end{split}
\end{align}
Similarly,
\begin{align}\label{J1222}
	\begin{split}
		|J_{1222}|
		&\lesssim\|\partial_2^k\Omega_1^+\|_{L^\infty_{x_1}L^2_{x_2}}^2
		\Big(\|\partial_1\Omega_1^+\|_{L^\infty}\|\Omega_1^-\|_{H^2}^2\\
		&\qquad\qquad+\|\partial_1\partial_2\Omega_1^-\|_{L^\infty_{x_1}L^2_{x_2}}
		\|\partial_2\Omega_1^+\|_{L^2_{x_1}L^\infty_{x_2}}\|\Omega_1^-\|_{H^2}\Big)\\
		&\lesssim\big(\|\partial_1\Omega^+\|_{H^2}^2+\|\partial_1\Omega^-\|_{H^2}^2\big)
		\big(\|\Omega^+\|_{H^2}^3+\|\Omega^-\|_{H^2}^3\big).
	\end{split}
\end{align}
For $J_{1223}$,
\begin{align}\label{J1223}
	|J_{1223}|
	\lesssim\|\partial_1\partial_2^k\Omega_1^+\|_{L^\infty_{x_1}L^2_{x_2}}
	\|\partial_2^k\Omega_1^+\|_{L^\infty_{x_1}L^2_{x_2}}
	\|\Omega_1^+\|_{L^\infty}\|\Omega_1^-\|_{H^2}^2
	\lesssim\|\partial_1\Omega_1^+\|_{H^2}^2\big(\|\Omega_1^+\|_{H^2}^3+\|\Omega_1^-\|_{H^2}^3\big).
\end{align}
Finally,
\begin{align}\label{J1224}
	|J_{1224}|
	\lesssim\|\Omega_2^+\|_{H^2}\int_{\mathbb R^2}|\partial_2^k\Omega_1^+|^2|\partial_2^2\Omega_1^-|^2\,dx.
\end{align}

For $J_{123}$, using
\[
-\Delta p=\nabla\cdot\nabla\cdot(\Omega^-\otimes\Omega^+),
\]
we obtain
\begin{align}\label{J123}
	\begin{split}
		|J_{123}|
		&\lesssim\|\partial_2^k\Omega_1^+\|_{L^\infty_{x_1}L^2_{x_2}}^2
		\|\partial_1\partial_2^2(\Omega^-\otimes\Omega^+)\|_{L^2}\|\partial_2^2\Omega_1^-\|_{L^2}\\
		&\lesssim\|\partial_2^k\Omega_1^+\|_{L^2}\|\partial_1\partial_2^k\Omega_1^+\|_{L^2}
		\big(\|\partial_1\Omega^-\|_{H^2}\|\Omega^+\|_{H^2}
		+\|\partial_1\Omega^+\|_{H^2}\|\Omega^-\|_{H^2}\big)\|\partial_2^2\Omega_1^-\|_{L^2}\\
		&\lesssim\big(\|\partial_1\Omega^+\|_{H^2}^2+\|\partial_1\Omega^-\|_{H^2}^2\big)
		\big(\|\Omega^+\|_{H^2}^3+\|\Omega^-\|_{H^2}^3\big).
	\end{split}
\end{align}

For $J_{124}$, integrating by parts in $x_1$ gives
\[
\begin{aligned}
	J_{124}={}&\int_{\mathbb R^2}\int_{-\infty}^{x_2}|\partial_2^k\Omega_1^+|^2\,dy_2\,
	\big(\nu_+\partial_1\partial_2^2\Omega^-_1+\nu_-\partial_1\partial_2^2\Omega^+_1\big)\partial_1\partial_2^2\Omega_1^-\,dx\\
	&+\int_{\mathbb R^2}\int_{-\infty}^{x_2}\partial_1|\partial_2^k\Omega_1^+|^2\,dy_2\,
	\big(\nu_+\partial_1\partial_2^2\Omega^-_1+\nu_-\partial_1\partial_2^2\Omega^+_1\big)\partial_2^2\Omega_1^-\,dx.
\end{aligned}
\]
By \eqref{nu and mu} and anisotropic interpolation,
\begin{align}\label{J124}
	\begin{split}
		|J_{124}|
		&\lesssim\|\partial_2^k\Omega_1^+\|_{L^\infty_{x_1}L^2_{x_2}}^2
		\big(\nu\|\partial_1\partial_2^2u\|_{L^2}+\mu\|\partial_1\partial_2^2b\|_{L^2}\big)
		\|\partial_1\partial_2^2\Omega_1^-\|_{L^2}\\
		&\quad+\|\partial_2^k\Omega_1^+\|_{L^\infty_{x_1}L^2_{x_2}}
		\|\partial_1\partial_2^k\Omega_1^+\|_{L^\infty_{x_1}L^2_{x_2}}
		\big(\nu\|\partial_1\partial_2^2u\|_{L^2}+\mu\|\partial_1\partial_2^2b\|_{L^2}\big)
		\|\partial_2^2\Omega_1^-\|_{L^2}\\
		&\lesssim\big(\nu\|\partial_1u\|_{H^2}+\mu\|\partial_1b\|_{H^2}\big)
		\|(\partial_1\Omega^+|\partial_1\Omega^-)\|_{H^2}
		\|(\Omega^+|\Omega^-)\|_{H^2}^2.
	\end{split}
\end{align}

Combining \eqref{J_{11}} and \eqref{J1221}--\eqref{J124}, and integrating in time, we obtain (the term $E_0^2(t)$ from \eqref{J_{11}} is absorbed into $\nu_M^{\frac12}\nu_m^{-\frac12}E_0^2(t)$ since $\nu_M\geq\nu_m$)
\begin{align}\label{J1}
	\left|\int_0^t(J_1(\tau) - J_{121}(\tau))\,d\tau\right|
	\lesssim \nu_M^{\frac12}\nu_m^{-\frac12}E_0^2(t)
	+\nu_m^{-1}E_0^{\frac52}(t)+E_0^{\frac12}(t)E_1^+(t).
\end{align}

We next estimate $J_2$. For $k=0$, using $\nabla\cdot\Omega^-=0$ and integrating by parts,
\begin{align}\label{J21}
	\begin{split}
		J_2
		&=\frac12\int_{\mathbb R^2}\int_{-\infty}^{x_2}\Omega^-\cdot\nabla|\Omega_1^+|^2\,dy_2\,|\partial_2^2\Omega_1^-|^2\,dx\\
		&=\frac12\int_{\mathbb R^2}\Omega_2^-|\Omega_1^+|^2|\partial_2^2\Omega_1^-|^2\,dx
		-\int_{\mathbb R^2}\int_{-\infty}^{x_2}\Omega_1^-|\Omega_1^+|^2\,dy_2\,
		\partial_1\partial_2^2\Omega_1^-\,\partial_2^2\Omega_1^-\,dx\\
		&\lesssim\|\Omega_2^-\|_{H^2}\int_{\mathbb R^2}|\Omega_1^+|^2|\partial_2^2\Omega_1^-|^2\,dx
		+\|(\Omega^+|\Omega^-) \|_{H^2}^3\|(\partial_1\Omega^+|\partial_1\Omega^-)\|_{H^2}^2.
	\end{split}
\end{align}

For $k=1$, the same argument yields
\begin{align}\label{J22}
	\begin{split}
		J_2
		={}&\frac12\int_{\mathbb R^2}\Omega_2^-|\partial_2\Omega_1^+|^2|\partial_2^2\Omega_1^-|^2\,dx
		-\int_{\mathbb R^2}\int_{-\infty}^{x_2}\Omega_1^-|\partial_2\Omega_1^+|^2\,dy_2\,
		\partial_1\partial_2^2\Omega_1^-\,\partial_2^2\Omega_1^-\,dx\\
		&+\int_{\mathbb R^2}\int_{-\infty}^{x_2}\partial_2\Omega_2^-|\partial_2\Omega_1^+|^2\,dy_2\,|\partial_2^2\Omega_1^-|^2\,dx\\
		&+\int_{\mathbb R^2}\int_{-\infty}^{x_2}\partial_2\Omega_1^-\,\partial_1\Omega_1^+\,\partial_2\Omega_1^+\,dy_2\,|\partial_2^2\Omega_1^-|^2\,dx\\
		\lesssim{}&\|\Omega_2^-\|_{H^2}\int_{\mathbb R^2}|\partial_2\Omega_1^+|^2|\partial_2^2\Omega_1^-|^2\,dx
		+\|(\Omega^+|\Omega^-) \|_{H^2}^3\|(\partial_1\Omega^+|\partial_1\Omega^-)\|_{H^2}^2.
	\end{split}
\end{align}
Consequently,
\begin{align}\label{J2}
	\int_0^t|J_2(\tau)|\,d\tau
	\lesssim E_0^{\frac12}(t)E_1^+(t)+\nu_m^{-1}E_0^{\frac52}(t).
\end{align}

For $J_3$, using the pressure equation and anisotropic interpolation,
\[
\begin{aligned}
	|J_3|
	&\lesssim\|\partial_1\partial_2^kp\|_{L^\infty_{x_1}L^2_{x_2}}
	\|\partial_2^k\Omega_1^+\|_{L^\infty_{x_1}L^2_{x_2}}\|\partial_2^2\Omega_1^-\|_{L^2}^2\\
	&\lesssim\|\partial_1\partial_2^kp\|_{L^2}^{\frac12}
	\|\partial_1^2\partial_2^kp\|_{L^2}^{\frac12}
	\|\partial_2^k\Omega_1^+\|_{L^2}^{\frac12}
	\|\partial_1\partial_2^k\Omega_1^+\|_{L^2}^{\frac12}\|\partial_2^2\Omega_1^-\|_{L^2}^2\\
	&\lesssim\|(\Omega^+|\Omega^-)\|_{H^2}^3
	\|(\partial_1\Omega^+|\partial_1\Omega^-)\|_{H^2}^2.
\end{aligned}
\]
Hence,
\begin{align}\label{J3}
	\int_0^t|J_3(\tau)|\,d\tau\lesssim\nu_m^{-1}E_0^{\frac52}(t).
\end{align}

Finally, integrating by parts in $x_1$ in $J_4$ gives
\[
\begin{aligned}
	J_4={}&\int_{\mathbb R^2}\int_{-\infty}^{x_2}
	\big(\nu_+\partial_1\partial_2^k\Omega^+_1+\nu_-\partial_1\partial_2^k\Omega^-_1\big)
	\partial_1\partial_2^k\Omega_1^+\,dy_2\,|\partial_2^2\Omega_1^-|^2\,dx\\
	&+\int_{\mathbb R^2}\int_{-\infty}^{x_2}
	\big(\nu_+\partial_1\partial_2^k\Omega^+_1+\nu_-\partial_1\partial_2^k\Omega^-_1\big)
	\partial_2^k\Omega_1^+\,dy_2\,\partial_1|\partial_2^2\Omega_1^-|^2\,dx .
\end{aligned}
\]
Using \eqref{nu and mu} and H\"older's inequality,
\[
\begin{aligned}
	|J_4|
	\lesssim{}&
	\big(\nu\|\partial_1\partial_2^ku\|_{L^\infty_{x_1}L^2_{x_2}}
	+\mu\|\partial_1\partial_2^kb\|_{L^\infty_{x_1}L^2_{x_2}}\big)
	\|\partial_1\partial_2^k\Omega^+\|_{L^\infty_{x_1}L^2_{x_2}}\|\partial_2^2\Omega_1^-\|_{L^2}^2\\
	&+\big(\nu\|\partial_1\partial_2^ku\|_{L^\infty_{x_1}L^2_{x_2}}
	+\mu\|\partial_1\partial_2^kb\|_{L^\infty_{x_1}L^2_{x_2}}\big)
	\|\partial_2^k\Omega^+\|_{L^\infty_{x_1}L^2_{x_2}}
	\|\partial_1\partial_2^2\Omega_1^-\|_{L^2}\|\partial_2^2\Omega_1^-\|_{L^2}\\
	\lesssim{}&
	\big(\nu\|\partial_1u\|_{H^2}+\mu\|\partial_1b\|_{H^2}\big)
	\|(\partial_1\Omega^+|\partial_1\Omega^-)\|_{H^2}\|(\Omega^+|\Omega^-)\|_{H^2}^2.
\end{aligned}
\]
Thus,
\begin{align}\label{J4}
	\int_0^t|J_4(\tau)|\,d\tau
	\lesssim\nu_M^{\frac12}\nu_m^{-\frac12}E_0^2(t).
\end{align}

Integrating \eqref{key-e1} in time, summing over $k=0,1$, and
combining \eqref{J1}--\eqref{J4}, we obtain
\[
E_1^+(t)
\lesssim
\nu_M^{\frac12}\nu_m^{-\frac12}E_0^2(t)
+\nu_m^{-1}E_0^{\frac52}(t)
+E_0^{\frac12}(t)E_1^+(t).
\]
Since $E_0(t)\leq E(t)$ and $E_1^+(t)\leq\nu_m E(t)$, it follows that
\[
E_1^+(t)
\lesssim
\frac{\nu_M^{\frac12}}{\nu_m^{\frac12}}E^2(t)
+\nu_m^{-1}E^{\frac52}(t)
+\nu_mE^{\frac32}(t).
\]
By symmetry, the same estimate holds for $E_1^-(t)$. This proves
\eqref{E1^{pm}(t)}.

\subsection{Proof of Theorem \ref{thm}}\label{subsec:close2}

Combining \eqref{E_0(t)} and \eqref{E1^{pm}(t)} with the definition of
$E(t)$, we obtain
\[
\begin{aligned}
	E(t)\leq{}&
	C E(0)
	+C\frac{1+\nu_M^{\frac12}\nu_m^{\frac12}}{\nu_m}
	E^{\frac32}(t)+C\frac{\nu_M^{\frac12}}{\nu_m^{\frac32}}E^2(t)
	+C\nu_m^{-2}E^{\frac52}(t).
\end{aligned}
\]
Here we used $\nu_m^{-1}\cdot\nu_mE^{\frac32}=E^{\frac32}\leq
\nu_M^{\frac12}\nu_m^{-\frac12}E^{\frac32}$ and
$\nu_m^{-1}E^2\leq\nu_M^{\frac12}\nu_m^{-\frac32}E^2$, both of which
follow from $\nu_M\geq\nu_m$. Set
\[
\delta_0
=\min\left\{
\nu_m^2,\frac{\nu_m}{\nu_M}
\right\}.
\]
By the definition of $\delta_0$,
\[
\frac{1+\nu_M^{\frac12}\nu_m^{\frac12}}{\nu_m}
\delta_0^{\frac12}
+\frac{\nu_M^{\frac12}}{\nu_m^{\frac32}}\delta_0
+\nu_m^{-2}\delta_0^{\frac32}
\lesssim 1.
\]
Assume that
\[
E(t)\leq c_1\delta_0
\qquad\text{for }0\leq t\leq T.
\]
If the absolute constant $c_1>0$ is sufficiently small, the nonlinear
terms in the preceding inequality can be absorbed, and hence
\[
E(t)\leq 2C E(0),
\qquad 0\leq t\leq T.
\]
Consequently, choosing the absolute constant $c_0>0$ sufficiently small,
Remark~\ref{rem:local} and a standard continuity argument give
\[
E(t)\leq 2C E(0),
\qquad t\geq0,
\]
provided that
\[
E(0)\leq c_0^2\delta_0.
\]
Since
\[
E(0)=2\big(\|u_0\|_{H^2}^2+\|b_0\|_{H^2}^2\big),
\]
after decreasing $c_0$ if necessary, the normalized smallness condition
can be written as
\[
\|u_0\|_{H^2}+\|b_0\|_{H^2}
\leq
c_0\min\left\{
\nu_m,
\frac{\nu_m^{\frac12}}{\nu_M^{\frac12}}
\right\}.
\]

Returning to the original variables and recalling that
\[
\widetilde\nu_m=\frac{\nu_m}{\gamma},
\qquad
\widetilde\nu_M=\frac{\nu_M}{\gamma},
\qquad
(v_0,h_0)=\frac1\gamma(u_0,b_0),
\]
we find that the normalized smallness condition becomes
\[
\|u_0\|_{H^2}+\|b_0\|_{H^2}
\leq
c_0\min\left\{
\nu_m,
\frac{\nu_m^{\frac12}}{\nu_M^{\frac12}}\gamma
\right\}.
\]
This is precisely the smallness condition in Theorem \ref{thm}.
Scaling back also yields the global bound stated in Theorem \ref{thm}, which completes the proof.

\section{Proof of Theorem \ref{cor}}\label{COR}

We again work in the normalized setting $\gamma=1$. The estimate
\eqref{E1^{pm}(t)} obtained in the previous section remains unchanged.
The improvement leading to Theorem \ref{cor} consists in replacing the
direct estimate \eqref{N_1234} for $N_1-N_4$ by a sharper estimate
based on the div--curl structure generated by the propagation in the
$x_2$-direction. The estimates \eqref{N_5} and \eqref{N_6} for $N_5$ and
$N_6$ will be used without modification. More precisely, we shall prove
\begin{align}
	\begin{split}
		E_0(t)
		\leq{}&
		C E(0)
		+C\nu_m^{-1}\big(E_1^+(t)+E_1^-(t)\big)\\
		&+
		C\frac{
			\nu_M^{\frac14}
			+\nu_M^{\frac12}\nu_m^{\frac14}}
		{\nu_m^{\frac34}}
		E^{\frac32}(t)
		+C\nu_m^{-1}E^2(t)
		+C\nu_m^{-1}E^{\frac74}(t).
	\end{split}
	\label{E_0(t)-new}
\end{align}
Compared with \eqref{E_0(t)}, the factor $\nu_m^{-1}$ in front of the
leading cubic term is replaced by $\nu_M^{\frac14}\nu_m^{-\frac34}$, at
the expense of the higher-order term $\nu_m^{-1}E^{\frac74}(t)$.

The section is organized as follows. In
Section~\ref{Estimate of energy E_0(t)-new} we prove \eqref{E_0(t)-new}.
We decompose $N_1+N_2$ into the terms $N_{11},\dots,N_{18}$ and bound each
of them by the energy dissipation times a space--time norm of the form
$\big\|\|K\Omega^\pm\|_{L^2_{x_1}}\|L\Omega^\mp\|_{L^2_{x_1}}\big\|_{L^2_tL^2_{x_2}}$
for a suitable pair $(L,K)$ of differential operators. To control these
norms, we derive, following \cite{JiangRenZhou2026}, two balance laws in
the $(t,x_2)$-plane for quadratic quantities of $(u,b)$ integrated in
$x_1$, apply the div--curl lemma (Lemma~\ref{div-curl lemma}), and
estimate the resulting source terms using only the horizontal
dissipation. The term $N_{15}$ requires an additional integration by
parts, and $N_3+N_4$ are treated by symmetry. In
Section~\ref{subsec:close3} we combine \eqref{E_0(t)-new} with
\eqref{E1^{pm}(t)}, close the estimate by a continuity argument, and
return to general $\gamma$, which gives Remark~\ref{rem:general},
Theorem~\ref{cor} and Corollary~\ref{cor1}.

\subsection{Refined estimate for $E_0(t)$}
\label{Estimate of energy E_0(t)-new}

\begin{proof}[Proof of \eqref{E_0(t)-new}]
Starting from \eqref{E0-preliminary}, we focus on the terms
$N_1-N_4$. We first consider $N_1+N_2$. Using $\nabla\cdot\Omega^\pm=0$ and the cancellation of the highest-order transport terms, we have
	\[
	\begin{aligned}
		N_1+N_2={}&(\partial_1^2\Omega^-_2\partial_2\Omega^+,\partial_1^2\Omega^+)_{L^2}
		+2(\partial_1\Omega^-_2\partial_2\partial_1\Omega^+,\partial_1^2\Omega^+)_{L^2}\\
		&+(\partial_1^2\Omega^-_1\partial_1\Omega^+,\partial_1^2\Omega^+)_{L^2}
		+2(\partial_1\Omega^-_1\partial_1^2\Omega^+,\partial_1^2\Omega^+)_{L^2}\\
		&+(\partial_2^2\Omega^-_1\partial_1\Omega^+_2,\partial_2^2\Omega^+_2)_{L^2}
		+2(\partial_2\Omega^-_1\partial_1\partial_2\Omega^+_2,\partial_2^2\Omega^+_2)_{L^2}\\
		&+(\partial_2^2\Omega^-_2\partial_2\Omega^+_2,\partial_2^2\Omega^+_2)_{L^2}
		+2(\partial_2\Omega^-_2\partial_2^2\Omega^+_2,\partial_2^2\Omega^+_2)_{L^2}\\
		\triangleq{}&N_{11}+N_{12}+N_{13}+N_{14}+N_{15}+N_{16}+N_{17}+N_{18}.
	\end{aligned}
	\]
	
	We first estimate $N_{11}$. By H\"older's inequality and anisotropic interpolation,
	\[
	\begin{aligned}
		\left|\int_0^tN_{11}(\tau)\,d\tau\right|
		&\lesssim
		\int_0^t
		\left\|
		\|\partial_1^2\Omega^-\|_{L^2_{x_1}}
		\|\partial_2\Omega^+\|_{L^2_{x_1}}
		\right\|_{L^2_{x_2}}
		\|\partial_1^2\Omega^+\|_{L^2_{x_2}L^\infty_{x_1}}\,d\tau\\
		&\lesssim
		\left\|
		\|\partial_1^2\Omega^-\|_{L^2_{x_1}}
		\|\partial_2\Omega^+\|_{L^2_{x_1}}
		\right\|_{L^2_tL^2_{x_2}}
		\|\partial_1\Omega^+\|_{L^2_tH^2_x}.
	\end{aligned}
	\]
	It therefore remains to control the first factor.
	
	Although the energy framework is formulated in the Els\"asser
	variables, the div--curl structure is more transparent in the velocity
	and magnetic variables $(u,b)$. We use these variables to derive two
	balance laws that control the mixed $(\Omega^+,\Omega^-)$ quantity above.
This construction is taken from the bilinear estimate of Jiang, Ren and
Zhou \cite[Lemma~3.4]{JiangRenZhou2026}, who consider the MHD system with
dissipation in all directions. Here only the horizontal dissipation is
available, and this is why we restrict ourselves to operators $L$
containing $\partial_1$; see the estimates of $I_{112}$--$I_{115}$
below.

Let $(L,K)$ be one of the following pairs of differential
	operators:
	\[
	(\partial_1^2,\partial_2),\qquad
	(\partial_1^2,\partial_1),\qquad
	(\partial_1\partial_2,\partial_1),\qquad
	(\partial_1,\partial_2).
	\]
	We derive a common space--time estimate for these four pairs. Fix one
	pair and suppress its dependence in the notation below.

		Applying $L$ to the equations for $u$ and $b$, taking the inner products with $Lu$ and $Lb$, respectively, and adding the resulting identities, we obtain
	\[
	\frac12\partial_t\big(|Lu|^2+|Lb|^2\big)
	-\partial_2(Lu \cdot Lb)
	=\widetilde G_1,
	\]
	where
	\[
	\begin{aligned}
	\widetilde G_1\triangleq{}&
	L(b\cdot\nabla b-u\cdot\nabla u-\nabla p)\cdot Lu
	+L(b\cdot\nabla u-u\cdot\nabla b)\cdot Lb\\
	&+\nu\partial_1^2 Lu\cdot  Lu
	+\mu\partial_1^2 L b\cdot Lb.
	\end{aligned}
	\]
	Integrating in $x_1$ gives
	\[
	\partial_t\left(\frac12\int_{\mathbb R_{x_1}}\big(|Lu|^2+|Lb|^2\big)\,dx_1\right)
	+\partial_2\left(-\int_{\mathbb R_{x_1}}L u\cdot L b\,dx_1\right)
	=\int_{\mathbb R_{x_1}}\widetilde G_1\,dx_1.
	\]
	
	Similarly, applying $K$ to the equations for $u$ and $b$, taking the inner products with $Kb$ and $Ku$, respectively, and adding, we obtain
	\[
	\partial_t(Ku\cdot Kb)
	-\frac12\partial_2\big(|Ku|^2+|Kb|^2\big)
	=\widetilde G_2,
	\]
	where
	\[
	\begin{aligned}
	\widetilde G_2\triangleq{}&
	K(b\cdot\nabla b-u\cdot\nabla u-\nabla p)\cdot Kb
	+K(b\cdot\nabla u-u\cdot\nabla b)\cdot Ku\\
	&+\nu\partial_1^2Ku\cdot Kb
	+\mu\partial_1^2 K b\cdot Ku.
	\end{aligned}
	\]
	Hence
	\[
	\partial_t\left(\int_{\mathbb R_{x_1}}Ku\cdot Kb\,dx_1\right)
	-\partial_2\left(\frac12\int_{\mathbb R_{x_1}}\big(| K u|^2+| Kb|^2\big)\,dx_1\right)
	=\int_{\mathbb R_{x_1}}\widetilde G_2\,dx_1.
	\]
	
	Set
	\[
	\begin{aligned}
	f_{11}&\triangleq\frac12\int_{\mathbb R_{x_1}}\big(|Lu|^2+|Lb|^2\big)\,dx_1,
	&\qquad
	f_{12}&\triangleq-\int_{\mathbb R_{x_1}}Lu\cdot Lb\,dx_1,\\
	f_{21}&\triangleq\int_{\mathbb R_{x_1}}Ku\cdot Kb\,dx_1,
	&\qquad
	f_{22}&\triangleq\frac12\int_{\mathbb R_{x_1}}\big(|Ku|^2+|Kb|^2\big)\,dx_1,
	\end{aligned}
	\]
	and
	\[
	G_1\triangleq\int_{\mathbb R_{x_1}}\widetilde G_1\,dx_1,
	\qquad
	G_2\triangleq\int_{\mathbb R_{x_1}}\widetilde G_2\,dx_1.
	\]
By integration by parts in the $x_1$-direction,
	\[
	\begin{aligned}
		G_1={}&\int_{\mathbb R_{x_1}}
		\Bigl[L(b\cdot\nabla b-u\cdot\nabla u-\nabla p)\cdot Lu
		+L(b\cdot\nabla u-u\cdot\nabla b)\cdot Lb\\
		&-\nu|\partial_1Lu|^2
		-\mu|\partial_1Lb|^2\Bigr]\,dx_1.
	\end{aligned}
	\]
	The diffusive part of $G_2$ is similarly given by
	\[
	G_2^{\mathrm d}
	=-(\nu+\mu)\int_{\mathbb R_{x_1}}
	\partial_1Ku\cdot\partial_1Kb\,dx_1.
	\]
	Then
	\[
	\partial_tf_{11}+\partial_2f_{12}=G_1,
	\qquad
	\partial_tf_{21}-\partial_2f_{22}=G_2.
	\]
	A direct computation gives
	\[
	\frac18\left(
	\|K\Omega^+\|_{L^2_{x_1}}^2\|L\Omega^-\|_{L^2_{x_1}}^2
	+\|K\Omega^-\|_{L^2_{x_1}}^2\|L\Omega^+\|_{L^2_{x_1}}^2
	\right)
	=f_{11}f_{22}+f_{12}f_{21}.
	\]
	In particular, $f_{11}f_{22}+f_{12}f_{21}\geq0$, and we can apply Lemma \ref{div-curl lemma} with $y=x_2$.

	To estimate the source terms, write
	\[
	G_i=G_i^{\mathrm{nl}}+G_i^{\mathrm{d}},\qquad i=1,2,
	\]
	where the superscripts denote the nonlinear and diffusive parts, respectively, and set
	\[
	F_{11}(x_2,\tau)=\int_{-\infty}^{x_2}f_{11}(z,\tau)\,dz,
	\qquad
	F_{21}(x_2,\tau)=\int_{x_2}^{+\infty}f_{21}(z,\tau)\,dz.
	\]
	Lemma \ref{div-curl lemma} then yields
	\[
	\begin{aligned}
		&\quad\left\|
		\|K\Omega^+\|_{L^2_{x_1}}
		\|L\Omega^-\|_{L^2_{x_1}}
		\right\|_{L^2_tL^2_{x_2}}^2
		+
		\left\|
		\|K\Omega^-\|_{L^2_{x_1}}
		\|L\Omega^+\|_{L^2_{x_1}}
		\right\|_{L^2_tL^2_{x_2}}^2\\
		&\lesssim I_{111}+I_{112}+I_{113}+I_{114}+I_{115},
	\end{aligned}
	\]
	where
	\[
	\begin{aligned}
		I_{111}&\triangleq\sup_{0\leq\tau\leq t}\|f_{11}(\tau)\|_{L^1_{x_2}}\|f_{21}(\tau)\|_{L^1_{x_2}},\\
		I_{112}&\triangleq\left|\int_0^t\int_{\mathbb R}F_{11}G_2^{\mathrm{nl}}\,dx_2\,d\tau\right|,
		\qquad
		I_{113}\triangleq\left|\int_0^t\int_{\mathbb R}F_{11}G_2^{\mathrm{d}}\,dx_2\,d\tau\right|,\\
		I_{114}&\triangleq\left|\int_0^t\int_{\mathbb R}F_{21}G_1^{\mathrm{nl}}\,dx_2\,d\tau\right|,
		\qquad
		I_{115}\triangleq\left|\int_0^t\int_{\mathbb R}F_{21}G_1^{\mathrm{d}}\,dx_2\,d\tau\right|.
	\end{aligned}
	\]
	The first term is bounded by
	\begin{align}\label{I_111}
		I_{111}
		\lesssim
		\|(u|b)\|_{L^\infty_tH^2_x}^4.
	\end{align}

We next estimate the nonlinear part of $G_2$. Since $K$ is
a first-order operator, the standard product estimates in two dimensions
and the pressure equation
\[
-\Delta p=\nabla\cdot\nabla\cdot(u\otimes u-b\otimes b)
\]
give
\[
\begin{aligned}
&\|K(b\cdot\nabla b-u\cdot\nabla u-\nabla p)\|_{L^2}
+\|K(b\cdot\nabla u-u\cdot\nabla b)\|_{L^2}\lesssim \|(u|b)\|_{H^2_x}^2.
\end{aligned}
\]
Consequently,
\[
\|G_2^{\mathrm{nl}}\|_{L^1_{x_2}}
\lesssim \|(u|b)\|_{H^2_x}^3.
\]
Moreover,
\[
\|F_{11}\|_{L^\infty_{x_2}}
\leq \|f_{11}\|_{L^1_{x_2}}
\lesssim \|(Lu|Lb)\|_{L^2_x}^2.
\]
Because every admissible $L$ contains at least one $\partial_1$, it
follows that
\begin{align}\label{I_112}
I_{112}
&\lesssim
\int_0^t
\|(Lu|Lb)\|_{L^2_x}^2
\|(u|b)\|_{H^2_x}^3\,d\tau\lesssim
\|(u|b)\|_{L^\infty_tH^2_x}^3
\|(\partial_1u|\partial_1b)\|_{L^2_tH^2_x}^2.
\end{align}

For $G_1^{\mathrm{nl}}$, each admissible $L$ has order at most two
and contains at least one $\partial_1$. Hence, for smooth vector fields
$v$ and $w$,
\[
\|L(v\cdot\nabla w)\|_{L^2}
\lesssim
\|v\|_{H^2}\|\partial_1w\|_{H^2}
+\|\partial_1v\|_{H^2}\|w\|_{H^2}.
\]
The pressure equation also gives
\[
\|L\nabla p\|_{L^2}
\lesssim
\|(u|b)\|_{H^2_x}
\|(\partial_1u|\partial_1b)\|_{H^2_x}.
\]
Therefore,
\[
\|G_1^{\mathrm{nl}}\|_{L^1_{x_2}}
\lesssim
\|(u|b)\|_{H^2_x}
\|(\partial_1u|\partial_1b)\|_{H^2_x}^2.
\]
Since
\[
\|F_{21}\|_{L^\infty_{x_2}}
\leq\|f_{21}\|_{L^1_{x_2}}
\lesssim\|(Ku|Kb)\|_{L^2_x}^2
\lesssim\|(u|b)\|_{H^2_x}^2,
\]
we obtain
\begin{align}\label{I_114}
I_{114}
&\lesssim
\int_0^t
\|(u|b)\|_{H^2_x}^3
\|(\partial_1u|\partial_1b)\|_{H^2_x}^2\,d\tau\lesssim
\|(u|b)\|_{L^\infty_tH^2_x}^3
\|(\partial_1u|\partial_1b)\|_{L^2_tH^2_x}^2.
\end{align}

For the diffusive part of $G_2$, the identity above for
$G_2^{\mathrm d}$ gives
\begin{align}\label{I_113}
I_{113}
\lesssim{}&
(\nu+\mu)\|(u|b)\|_{L^\infty_tH^2_x}^2
\|\partial_1u\|_{L^2_tH^2_x}
\|\partial_1b\|_{L^2_tH^2_x}.
\end{align}
For the diffusive part of $G_1$, we similarly have
\begin{align}\label{I_115}
I_{115}
\lesssim{}&
\|(u|b)\|_{L^\infty_tH^2_x}^2
\left(
\nu\|\partial_1u\|_{L^2_tH^2_x}^2
+\mu\|\partial_1b\|_{L^2_tH^2_x}^2
\right).
\end{align}

Combining \eqref{I_111}--\eqref{I_115} and using the
equivalence between the $(u,b)$ and $(\Omega^+,\Omega^-)$ norms, we
first obtain
\[
\begin{aligned}
&\left\|
\|K\Omega^+\|_{L^2_{x_1}}
\|L\Omega^-\|_{L^2_{x_1}}
\right\|_{L^2_tL^2_{x_2}}^2
+
\left\|
\|K\Omega^-\|_{L^2_{x_1}}
\|L\Omega^+\|_{L^2_{x_1}}
\right\|_{L^2_tL^2_{x_2}}^2\\
&\qquad\lesssim
E_0^2(t)+\nu_m^{-1}E_0^{\frac52}(t)
+\nu_M^{\frac12}\nu_m^{-\frac12}E_0^2(t).
\end{aligned}
\]
Returning to the estimate of $N_{11}$, we obtain
\[
\begin{aligned}
\left|\int_0^tN_{11}(\tau)\,d\tau\right|
\lesssim{}&
\nu_m^{-\frac12}E_0^{\frac32}(t)
+\nu_m^{-1}E_0^{\frac74}(t)
+\nu_M^{\frac14}\nu_m^{-\frac34}E_0^{\frac32}(t).
\end{aligned}
\]

Applying the preceding estimate with the operator pairs listed
above, and using the divergence-free identities where necessary, gives
the same bound for $N_{12}$, $N_{13}$, $N_{14}$, $N_{16}$, $N_{17}$,
and $N_{18}$.
More precisely, the relevant pairs are
\[
\begin{array}{c|c}
\text{terms} & (L,K)\\ \hline
N_{11},N_{16} & (\partial_1^2,\partial_2)\\
N_{13},N_{14} & (\partial_1^2,\partial_1)\\
N_{12},N_{17},N_{18} & (\partial_1\partial_2,\partial_1).
\end{array}
\]
The only term requiring a slight modification is $N_{15}$, since a direct estimate would involve the uncontrolled derivative $\partial_2^2\Omega^-_1$. Integrating by parts in $x_2$, we write
	\[
	N_{15}
	=-(\partial_2\Omega^-_1\partial_2\partial_1\Omega^+_2,\partial_2^2\Omega^+_2)_{L^2}
	-(\partial_2\Omega^-_1\partial_1\Omega^+_2,\partial_2^3\Omega^+_2)_{L^2}
	\triangleq N_{151}+N_{152}.
	\]
The term $N_{151}$ is treated in the same way as $N_{11}$,
	with $(L,K)=(\partial_1^2,\partial_2)$. For $N_{152}$, using
	\[
	\partial_2^3\Omega_2^+=-\partial_1\partial_2^2\Omega_1^+,
	\]
	together with anisotropic interpolation, we have
	\[
	\begin{aligned}
		&\quad\int_0^t|N_{152}(\tau)|\,d\tau\\
		&\lesssim
		\left\|
		\|\partial_2\Omega^-_1\|_{L^\infty_{x_1}}
		\|\partial_1\Omega^+_2\|_{L^2_{x_1}}
		\right\|_{L^2_tL^2_{x_2}}
		\|\partial_1\Omega^+\|_{L^2_tH^2_x}\\
		&\lesssim
		\left\|
		\|\partial_2\Omega^-_1\|_{L^2_{x_1}}
		\|\partial_1\Omega^+\|_{L^2_{x_1}}
		\right\|_{L^2_tL^2_{x_2}}^{\frac12}
		\left\|
		\|\partial_1\partial_2\Omega^-_1\|_{L^2_{x_1}}
		\|\partial_1\Omega^+\|_{L^2_{x_1}}
		\right\|_{L^2_tL^2_{x_2}}^{\frac12}
		\|\partial_1\Omega^+\|_{L^2_tH^2_x}\\
		&\lesssim
		\left(
		\left\|
		\|\partial_2\Omega^-_1\|_{L^2_{x_1}}
		\|\partial_1\Omega^+\|_{L^2_{x_1}}
		\right\|_{L^2_tL^2_{x_2}}
		+
		\left\|
		\|\partial_1\partial_2\Omega^-_1\|_{L^2_{x_1}}
		\|\partial_1\Omega^+\|_{L^2_{x_1}}
		\right\|_{L^2_tL^2_{x_2}}
		\right)
		\|\partial_1\Omega^+\|_{L^2_tH^2_x}.
	\end{aligned}
	\]
The two mixed norms on the right-hand side are controlled by
the preceding div--curl estimate together with the dissipative
$x_1$-derivatives. They correspond, respectively, to
$(L,K)=(\partial_1,\partial_2)$ and
$(L,K)=(\partial_1\partial_2,\partial_1)$.
	Consequently,
	\begin{align}\label{N_1,N_2}
		\int_0^t\big(|N_1(\tau)|+|N_2(\tau)|\big)\,d\tau
		\lesssim
		\nu_m^{-\frac12}E_0^{\frac32}(t)
		+\nu_m^{-1}E_0^{\frac74}(t)
		+\nu_M^{\frac14}\nu_m^{-\frac34}E_0^{\frac32}(t).
	\end{align}
	By symmetry,
	\begin{align}\label{N_3,N_4}
		\int_0^t\big(|N_3(\tau)|+|N_4(\tau)|\big)\,d\tau
		\lesssim
		\nu_m^{-\frac12}E_0^{\frac32}(t)
		+\nu_m^{-1}E_0^{\frac74}(t)
		+\nu_M^{\frac14}\nu_m^{-\frac34}E_0^{\frac32}(t).
	\end{align}
	
	Finally, combining \eqref{E0-preliminary}, \eqref{N_1,N_2}, \eqref{N_3,N_4}, \eqref{N_5}, and \eqref{N_6}, and using $\nu_M\geq\nu_m$, we obtain
	\[
	E_0(t)
	\lesssim
	E(0)
	+\nu_m^{-1}\big(E_1^+(t)+E_1^-(t)\big)
	+\frac{\nu_M^{\frac14}+\nu_M^{\frac12}\nu_m^{\frac14}}{\nu_m^{\frac34}}E^{\frac32}(t)
	+\nu_m^{-1}E^2(t)
	+\nu_m^{-1}E^{\frac74}(t).
	\]
	This is precisely \eqref{E_0(t)-new}.
\end{proof}

\subsection{Proof of Theorem \ref{cor} and Corollary \ref{cor1}}\label{subsec:close3}

We now close the refined estimate. Combining \eqref{E_0(t)-new} with
\eqref{E1^{pm}(t)} and the definition of $E(t)$ gives
\[
\begin{aligned}
E(t)\leq{}&
C E(0)
+C\frac{\nu_M^{\frac14}+\nu_M^{\frac12}\nu_m^{\frac14}}
{\nu_m^{\frac34}}E^{\frac32}(t)
+C\frac{\nu_M^{\frac12}}{\nu_m^{\frac32}}E^2(t)+C\nu_m^{-1}E^{\frac74}(t)
+C\nu_m^{-2}E^{\frac52}(t).
\end{aligned}
\]
Set
\[
\delta_1=\min\left\{
\nu_m^{\frac43},
\frac{\nu_m}{\nu_M},
\frac{\nu_m^{\frac32}}{\nu_M^{\frac12}}
\right\}.
\]
By the definition of $\delta_1$,
\[
\frac{\nu_M^{\frac14}+\nu_M^{\frac12}\nu_m^{\frac14}}
{\nu_m^{\frac34}}\delta_1^{\frac12}
+\frac{\nu_M^{\frac12}}{\nu_m^{\frac32}}\delta_1
+\nu_m^{-1}\delta_1^{\frac34}
+\nu_m^{-2}\delta_1^{\frac32}
\lesssim1.
\]
Hence, under the bootstrap assumption $E(t)\leq c_1\delta_1$ on
$[0,T]$, all nonlinear terms can be absorbed once $c_1>0$ is chosen
sufficiently small. As in the proof of Theorem \ref{thm}, a standard
continuity argument yields
\[
E(t)\lesssim E(0),\qquad t\geq0,
\]
provided $E(0)\leq c_0^2\delta_1$ for a sufficiently small $c_0>0$.
Equivalently, the normalized problem is globally controlled under
\begin{equation}\label{normalized refined threshold}
\|u_0\|_{H^2}+\|b_0\|_{H^2}
\leq c_0\min\left\{
\nu_m^{\frac23},
\frac{\nu_m^{\frac12}}{\nu_M^{\frac12}},
\frac{\nu_m^{\frac34}}{\nu_M^{\frac14}}
\right\}.
\end{equation}

We now return to the original variables. No restriction on the size of
$\nu$, $\mu$ or $\gamma$ has been used so far. Applying
\eqref{normalized refined threshold} to the scaled system
\eqref{scaling system}, with
\[
\widetilde\nu_m=\frac{\nu_m}{\gamma},
\qquad
\widetilde\nu_M=\frac{\nu_M}{\gamma},
\]
and scaling back, we obtain, for all $\nu,\mu,\gamma>0$, the sharper
condition
\begin{align}\label{general-gamma-sharp}
	\|u_0\|_{H^2}+\|b_0\|_{H^2}
	\leq
	c_0\min\left\{
	\nu_m^{\frac23}\gamma^{\frac13},
	\frac{\nu_m^{\frac12}}{\nu_M^{\frac12}}\gamma,
	\frac{\nu_m^{\frac34}}{\nu_M^{\frac14}}\gamma^{\frac12}
	\right\}.
\end{align}
This proves Remark \ref{rem:general}. If, in addition, $0<\nu,\mu\leq1$
and $\gamma\geq1$, then
\[
\frac{\nu_m^{\frac12}}{\nu_M^{\frac12}}\gamma
\geq
\nu_m^{\frac34}\gamma^{\frac12},
\qquad
\frac{\nu_m^{\frac34}}{\nu_M^{\frac14}}\gamma^{\frac12}
\geq
\nu_m^{\frac34}\gamma^{\frac12}.
\]
Therefore, \eqref{general-gamma-sharp} follows from
\[
\|u_0\|_{H^2}+\|b_0\|_{H^2}
\leq
c_0\min\left\{
\nu_m^{\frac23}\gamma^{\frac13},
\nu_m^{\frac34}\gamma^{\frac12}
\right\}.
\]
The global estimate follows from the normalized estimate by the same
scaling. This proves Theorem \ref{cor}.

Finally, suppose that
\[
\nu_M\leq\beta\nu_m^{\frac13}.
\]
Then
\begin{align*}
	\frac{\nu_m^{\frac12}}{\nu_M^{\frac12}}\gamma
	&\geq
	\beta^{-\frac12}\nu_m^{\frac13}\gamma
	\geq
	\beta^{-\frac12}\nu_m^{\frac23}\gamma^{\frac13},\\
	\frac{\nu_m^{\frac34}}{\nu_M^{\frac14}}\gamma^{\frac12}
	&\geq
	\beta^{-\frac14}\nu_m^{\frac23}\gamma^{\frac12}
	\geq
	\beta^{-\frac14}\nu_m^{\frac23}\gamma^{\frac13},
\end{align*}
where we used $\nu_m\leq1$ and $\gamma\geq1$. Consequently,
\eqref{general-gamma-sharp} holds provided that
\[
\|u_0\|_{H^2}+\|b_0\|_{H^2}
\leq
c_\beta\nu_m^{\frac23}\gamma^{\frac13},
\]
where $c_\beta>0$ is sufficiently small and depends only on $\beta$.
This proves Corollary \ref{cor1}.

\appendix
\refstepcounter{section}
\section*{Appendix \thesection. A div--curl estimate}\label{app:divcurl}

We record the div--curl estimate used in Section~\ref{COR}. It is
Lemma~2.3 in \cite{LaiShaoZhou2025} (see also \cite{WangZhou2023}), and it
was used for the MHD equations in \cite[Lemma~3.3]{JiangRenZhou2026}. We
include the short proof for the reader's convenience. The name reflects the fact that $(f_{11},f_{12})$ is
divergence free in the $(t,y)$-plane up to $G_1$, while $(f_{22},f_{21})$
is curl free up to $G_2$; the integrand $f_{11}f_{22}+f_{12}f_{21}$ is
their Euclidean inner product (compare \cite{Murat1978,Tartar1979}). The
functional $Q(t)$ in the proof is an interaction potential in the spirit
of \cite{Glimm1965}.

\begin{lem}\label{div-curl lemma}
	Let $T>0$. Suppose that $f_{ij}$, $i,j\in\{1,2\}$, and $G_i$,
	$i\in\{1,2\}$, are smooth real-valued functions on
	$\mathbb R\times[0,T]$, with sufficient decay and integrability to
	justify the integrations below. Assume that
	\begin{equation}\label{div-curl-system}
		\left\{
		\begin{aligned}
			\partial_t f_{11}+\partial_y f_{12}&=G_1,\\
			\partial_t f_{21}-\partial_y f_{22}&=G_2.
		\end{aligned}
		\right.
	\end{equation}
	Then
	\begin{align*}
	\left|\int_0^T\int_{\mathbb R}
		\bigl(f_{11}f_{22}+f_{12}f_{21}\bigr)\,dy\,dt\right|
		&\leq
		2\sup_{0\leq t\leq T}
		\|f_{11}(t)\|_{L^1(\mathbb R)}
		\|f_{21}(t)\|_{L^1(\mathbb R)}\\
		&\quad+
		\left|\int_0^T\int_{\mathbb R}
		\left(\int_{-\infty}^y f_{11}(z,t)\,dz\right)
		G_2(y,t)\,dy\,dt\right|
		\\
		&\quad+
		\left|\int_0^T\int_{\mathbb R}
		\left(\int_y^\infty f_{21}(z,t)\,dz\right)
		G_1(y,t)\,dy\,dt\right|.
	\end{align*}
\end{lem}

\begin{proof}
	Define
	\[
	Q(t)
	=
	\int_{\mathbb R}f_{11}(y,t)
	\left(\int_y^\infty f_{21}(z,t)\,dz\right)\,dy.
	\]
	Differentiating $Q(t)$ and using \eqref{div-curl-system}, we obtain
	\begin{align*}
		\frac{d}{dt}Q(t)
		={}&
		\int_{\mathbb R}
		\bigl(G_1-\partial_yf_{12}\bigr)(y,t)
		\left(\int_y^\infty f_{21}(z,t)\,dz\right)\,dy
		\\
		&+
		\int_{\mathbb R}f_{11}(y,t)
		\left(\int_y^\infty
		\bigl(G_2+\partial_zf_{22}\bigr)(z,t)\,dz\right)\,dy.
	\end{align*}
	Integration by parts in $y$ and Fubini's theorem then give
	\begin{align*}
		\frac{d}{dt}Q(t)
		={}&
		-\int_{\mathbb R}
		\bigl(f_{11}f_{22}+f_{12}f_{21}\bigr)(y,t)\,dy
		\\
		&+
		\int_{\mathbb R}
		\left(\int_{-\infty}^y f_{11}(z,t)\,dz\right)
		G_2(y,t)\,dy
		+
		\int_{\mathbb R}
		\left(\int_y^\infty f_{21}(z,t)\,dz\right)
		G_1(y,t)\,dy.
	\end{align*}
	Integrating this identity from $0$ to $T$ and taking absolute values
	produces the two source terms in the asserted estimate together with
	the boundary term $|Q(0)-Q(T)|$. For every $t\in[0,T]$,
	\[
	|Q(t)|
	\leq
	\|f_{11}(t)\|_{L^1(\mathbb R)}
	\|f_{21}(t)\|_{L^1(\mathbb R)}.
	\]
	Therefore,
	\[
	|Q(0)-Q(T)|
	\leq
	2\sup_{0\leq t\leq T}
	\|f_{11}(t)\|_{L^1(\mathbb R)}
	\|f_{21}(t)\|_{L^1(\mathbb R)},
	\]
	which completes the proof.
\end{proof}

\end{document}